\documentclass[11pt]{amsart}
\pdfoutput=1

\usepackage{amssymb}
\usepackage{microtype}
\usepackage{mathtools}
\usepackage{stmaryrd}
\usepackage{tikz}

\usepackage[numbers]{natbib}
\usepackage{enumitem}
\usepackage{aliascnt}

\usepackage{hyperref}
\hypersetup{
  pdftitle={Multisymplecticity of hybridizable discontinuous Galerkin methods},
  pdfauthor={Robert I. McLachlan and Ari Stern},
  pdfsubject={MSC 2010: 65N30, 37K05},
  bookmarksopen=false,
}

\theoremstyle{plain}
\newtheorem{theorem}{Theorem}[section]
\newaliascnt{lemma}{theorem}
\newtheorem{lemma}[lemma]{Lemma}
\aliascntresetthe{lemma}
\newaliascnt{proposition}{theorem}
\newtheorem{proposition}[proposition]{Proposition}
\aliascntresetthe{proposition}
\newaliascnt{corollary}{theorem}
\newtheorem{corollary}[corollary]{Corollary}
\aliascntresetthe{corollary}

\theoremstyle{definition}
\newaliascnt{definition}{theorem}
\newtheorem{definition}[definition]{Definition}
\aliascntresetthe{definition}
\newaliascnt{example}{theorem}
\newtheorem{example}[example]{Example}
\aliascntresetthe{example}
\newaliascnt{assumption}{theorem}

\aliascntresetthe{assumption}

\theoremstyle{remark}
\newaliascnt{remark}{theorem}
\newtheorem{remark}[remark]{Remark}
\aliascntresetthe{remark}
\newaliascnt{notation}{theorem}
\newtheorem{notation}[notation]{Notation}
\aliascntresetthe{notation}

\begin{document}
\author{Robert I.~McLachlan}
\address{IFS, Massey University\\
  Palmerston North, New Zealand 5301}
\email{r.mclachlan@massey.ac.nz}

\author{Ari Stern}
\address{Department of Mathematics\\
Washington University in St.~Louis\\
Campus Box 1146\\
One Brookings Drive\\
Saint Louis, Missouri 63130}
\email{stern@wustl.edu}

\title[Multisymplecticity of HDG methods]{Multisymplecticity of hybridizable discontinuous Galerkin methods}

\begin{abstract}
  In this paper, we prove necessary and sufficient conditions for a
  hybridizable discontinuous Galerkin (HDG) method to satisfy a
  multisymplectic conservation law, when applied to a
  canonical Hamiltonian system of partial differential equations. We
  show that these conditions are satisfied by the ``hybridized''
  versions of several of the most commonly-used finite element
  methods, including mixed, nonconforming, and discontinuous Galerkin
  methods. (Interestingly, for the continuous Galerkin method in
  dimension greater than one, we show that multisymplecticity only
  holds in a weaker sense.)  Consequently, these general-purpose
  finite element methods may be used for structure-preserving
  discretization (or semidiscretization) of canonical Hamiltonian
  systems of ODEs or PDEs. This establishes multisymplecticity for a
  large class of arbitrarily-high-order methods on unstructured
  meshes.
\end{abstract}


\maketitle

\section{Introduction}
\label{sec:intro}

\subsection{Motivation and background}

Hamiltonian systems of ordinary differential equations (ODEs) and
partial differential equations (PDEs) are ubiquitous in applications,
especially in the modeling of physical systems.

One essential property of a Hamiltonian ODE is that its time flow is
\emph{symplectic}: that is, it conserves a closed, nondegenerate
$2$-form on phase space. This has motivated the development of
\emph{symplectic integrators}: one-step numerical integrators that,
when applied to Hamiltonian ODEs, are also symplectic maps. It turns
out that these methods have several numerical advantages that result
from preserving the symplectic structure. Furthermore, most of these
methods (such as symplectic partitioned Runge--Kutta methods) may be
applied to general systems of ODEs, whether or not the user is aware
of any Hamiltonian/symplectic structure---but if such a structure is
present, then the methods will automatically preserve them.  For a
comprehensive survey of structure-preserving numerical integrators,
including symplectic integrators, see~\citet{HaLuWa2006}.

Similarly, Hamiltonian PDEs satisfy a \emph{multisymplectic
  conservation law}. Since symplecticity is a desirable property for
numerical integration of canonical Hamiltonian ODEs, it is natural to
seek numerical methods for canonical Hamiltonian PDEs whose solutions
satisfy the multisymplectic conservation law, in an appropriate sense.
There has been important work on multisymplectic methods over the past
two decades, particularly by Marsden and collaborators
\citep{MaPaSh1998,MaSh1999,MaPeShWe2001,LeMaOrWe2003} and Reich and
collaborators
\citep{Reich2000a,Reich2000b,BrRe2001,MoRe2003,FrMoRe2006}. However,
most of these methods have consisted either of tensor products of
symplectic Runge--Kutta-type methods on rectangular grids or of
relatively low-order, finite-difference-type methods on unstructured
meshes. For the variational integrators of Marsden et~al., one must
also know the (Lagrangian) geometric structure of the PDE, in advance,
in order to devise the method.

The impact of multisymplecticity on solutions to PDEs, and on their
discretizations, is not fully understood. However, multisymplecticity
is known to be necessary to preserve traveling waves of hyperbolic
equations \citep{mcdonald2016travelling}, and compact multisymplectic
methods can preserve dispersion relations much better than
non-multisymplectic methods or noncompact finite difference methods
\citep{ascher2004multisymplectic,mclachlan2011linear,mclachlan2014high}.
For boundary value problems, multisymplecticity restricts the types of
bifurcations that can occur
\citep{mclachlan2018preservation,mclachlan2018bifurcation}.  Because
it is a local property, multisymplecticity is a strictly stronger
property than the symplecticity obtained by integrating over space.

There has been some previous work on the application of finite element
methods to certain problems on structured (especially rectangular)
meshes. \citet{GuJiLiWu2001} considered the 2-{D} nonlinear Poisson
equation on a regular rectangular grid, meshed with biased triangles,
using the continuous Galerkin method with linear shape functions, and
they showed that the degrees of freedom satisfied a multisymplectic
finite-difference scheme. \citet{ZhBaLiWu2003} did the same for
first-order rectangular Lagrange elements on a regular
grid. \citet{Chen2008} used first- and second-order rectangular
elements to derive Lagrangian variational methods (in the sense of
\citet{MaPaSh1998}) and applied these to the sine-Gordon equation on a
regular grid, while pointing out that higher-order elements could also
be used in principle. In all of these examples, however, finite
elements were really only used as a tool to construct a
finite-difference stencil on a regular, 2-{D} rectangular grid.

As finite element methods are traditionally formulated, there is a
serious conceptual obstacle to discussing multisymplecticity. Namely,
many classical finite element methods are posed on spaces of
\emph{global} functions, making it difficult to make sense of
\emph{local} properties like the multisymplectic conservation
law. Indeed, the interpretation of multisymplecticity is much more
straightforward for finite difference or finite volume methods with
local stencils, or for tensor products of 1-{D} integrators on a
rectangular grid, which may explain why the previous work has been
focused on such methods.

Hybrid finite element methods provide a way around this obstacle,
since they consist of \emph{local problems coupled through their
  boundary traces}, where the boundary traces are allowed to be
independent variables. (Oftentimes, these boundary traces are
interpreted as ``Lagrange multipliers'' enforcing weak continuity
between local regions.) While hybrid methods have a long history (see
the comprehensive work by \citet{BrFo1991}), the recent work of
\citet{CoGoLa2009} has shown that a wide variety of finite element
methods---including those not previously thought of as hybrid
methods---may be ``hybridized'' within a unified framework. Such
methods are called hybridizable discontinuous Galerkin (HDG) methods,
and they include not only several classical mixed and hybrid methods,
but also hybridized versions of continuous and discontinuous Galerkin
methods, nonconforming methods, and others. This framework provides
precisely the local structure needed to examine multisymplecticity of
the finite element methods in this class.

\subsection{Organization of the paper} The paper is organized as follows:

\begin{itemize}
\item In \autoref{sec:pdes}, we review systems of PDEs in a particular
  canonical form. This form includes the de~Donder--Weyl equations for
  a Hamiltonian and many elliptic and hyperbolic variational PDEs. We
  recall the multisymplectic conservation law for classical (i.e.,
  smooth) solutions and illustrate how this manifests concretely for a
  class of semilinear elliptic PDEs in mixed form. We also discuss the
  relationship between multisymplecticity of solutions and reciprocity
  principles in physical systems.
  
\item In \autoref{sec:flux}, we develop a hybrid ``flux formulation''
  for these systems of PDEs. As in \citet{CoGoLa2009}, this yields a
  collection of weak problems on non-overlapping subdomains, coupled
  only through approximate traces on their shared boundaries. Our
  framework includes not only the linear second-order elliptic PDEs
  considered by \citet{CoGoLa2009}, but also a more general class of
  nonlinear systems of PDEs, including canonical Hamiltonian PDEs. 

  Within this formulation, we establish criteria for solutions to
  satisfy weak and strong versions of the multisymplectic conservation
  law; the distinction is shown to be related to weak and strong
  conservativity of numerical fluxes
  (cf.~\citet{ArBrCoMa2001,CoGoLa2009}). In addition to the subsequent
  numerical applications, we also use a domain-decomposition argument
  to write the weak problem (in the sense of distributions) in this
  flux formulation, thereby establishing multisymplecticity for weak
  solutions to Hamiltonian PDEs.
  
\item In \autoref{sec:hdg}, we examine several particular classes of
  HDG methods, including the hybridized Raviart--Thomas (RT-H),
  Brezzi--Douglas--Marini (BDM-H), local discontinuous Galerkin
  (LDG-H), continuous Galerkin (CG-H), nonconforming (NC-H), and
  interior penalty (IP-H) methods. These methods are posed in the
  framework of \autoref{sec:flux}, and their multisymplecticity is
  then examined. Each of these methods, except for CG-H, is proved to
  be strongly multisymplectic. A counterexample shows that CG-H is
  only weakly multisymplectic, resulting from the fact that it is only
  weakly conservative.
\end{itemize}

\section{Canonical and multisymplectic systems of PDEs}
\label{sec:pdes}

\subsection{Canonical systems of PDEs} Given a domain
$ U \subset \mathbb{R}^m $, consider a system of first-order PDEs
having the form
\begin{equation}
  \label{eqn:pdeSystem}
  \partial _\mu u ^i = \phi ^i _\mu ( \cdot , u , \sigma ) , \qquad - \partial _\mu \sigma _i ^\mu = f _i ( \cdot , u , \sigma ) ,
\end{equation}
where $ \mu = 1 , \ldots, m $ and $ i = 1 , \ldots , n $.  Here,
$ u = u ^i (x) $ and $ \sigma = \sigma _i ^\mu (x) $ are unknown
functions on $U$, while $ \phi = \phi ^i _\mu ( x, u , \sigma ) $ and
$ f = f _i ( x, u , \sigma ) $ are given functions on
$ U \times \mathbb{R}^n \times \mathbb{R} ^{ m n } $. We abbreviate
$ \partial _\mu \coloneqq \partial / \partial x ^\mu $ and adopt the
Einstein index convention of summing over repeated indices---so, for
instance, the expression $ \partial _\mu \sigma _i ^\mu $ in
\eqref{eqn:pdeSystem} has an implied sum over $\mu$ and may therefore
be interpreted as the divergence of $ \sigma _i $.

Among these is the important class of \emph{(canonical) Hamiltonian
  systems},
\begin{equation}
  \label{eqn:ddw}
  \partial _\mu u ^i = \frac{ \partial H }{ \partial \sigma _i ^\mu } , \qquad -\partial _\mu \sigma _i ^\mu = \frac{ \partial H }{ \partial u ^i } .
\end{equation}
where $ H = H ( x, u , \sigma ) $ is a function called the
\emph{Hamiltonian}. In the special case $ m = 1 $, the resulting
system of ODEs yields Hamilton's equations of classical mechanics,
which are usually written as
\begin{equation*}
  \dot{q} ^i = \frac{ \partial H }{ \partial p _i } , \qquad - \dot{p} _i = \frac{ \partial H }{ \partial q ^i } .
\end{equation*}
The equations \eqref{eqn:ddw} are called the \emph{de~Donder--Weyl
  equations} (\citet{deDonder1935,Weyl1935}). These canonical systems
are an important special case of a more general class of Hamiltonian
systems of PDEs, cf.~\citet{Bridges1996,Bridges1997}.

Throughout this section, we assume that all of the functions above are
smooth. Later, in \autoref{sec:flux}, we will relax this assumption in
order to introduce a weak formulation of \eqref{eqn:pdeSystem}.

\begin{example}[semilinear elliptic PDE]
  \label{ex:semilinearIntro}
  Let $ n = 1 $, so that $u = u (x) $ is a scalar field and
  $\sigma = \sigma ^\mu (x) $ a vector field on
  $U \subset \mathbb{R}^m $. Consider
  \begin{equation*}
    H ( x , u , \sigma ) = \frac{1}{2} a _{ \mu \nu } (x) \sigma ^\mu
    \sigma ^\nu + F (x , u) ,
  \end{equation*}
  where $ a = a ^{ \mu \nu } (x) $ is symmetric and positive-definite
  with matrix inverse
  $ a _{ \mu \nu } (x) \coloneqq \bigl( a ^{ \mu \nu } (x) \bigr)
  ^{-1} $ at each $ x \in U $, and where $F$ is arbitrary.  Then the
  de~Donder--Weyl equations for this Hamiltonian are
  \begin{equation*}
    \partial _\mu u = a _{ \mu \nu } \sigma ^\nu , \qquad -\partial _\mu \sigma ^\mu = \frac{ \partial F }{ \partial u } . 
  \end{equation*}
  From the first of these equations, we have
  $ \sigma ^\mu = a ^{ \mu \nu } \partial _\nu u $, so substituting
  this into the second yields
  \begin{equation*}
    - \partial _\mu a ^{ \mu \nu } \partial _\nu u = \frac{ \partial F }{ \partial u } ,
  \end{equation*}
  which is a second-order semilinear elliptic PDE in divergence
  form. Note that we could also have written the de~Donder--Weyl
  equations in the equivalent, coordinate-free form,
  \begin{equation*}
    \operatorname{grad} u = a ^{-1} \sigma , \qquad - \operatorname{div} \sigma = \frac{ \partial F }{ \partial u },
  \end{equation*}
  so the substitution $ \sigma = a \operatorname{grad} u $ yields
  \begin{equation*}
    - \operatorname{div} ( a \operatorname{grad} u ) = \frac{ \partial F }{ \partial u },
  \end{equation*}
  which is an equivalent expression for the second-order PDE above.

  An important special case is when
  $ F ( x, u ) = f (x) u - \frac{1}{2} c (x) u ^2 $ for given $f$ and
  $c$ on $U$. In this case, we obtain a linear second-order
  elliptic PDE,
  \begin{equation*}
    - \operatorname{div} ( a \operatorname{grad} u ) + c u = - \partial _\mu a ^{ \mu \nu } \partial _\nu u + c u = f .
  \end{equation*}
  In particular, if $ a ^{ \mu \nu } \equiv \delta ^{ \mu \nu } $
  (where $\delta$ is the Kronecker delta, i.e., $a$ is the identity
  matrix) and $ c \equiv 0 $, then this simply becomes Poisson's
  equation $ - \Delta u = f $ on $U$.

  We will regularly return to this example throughout the paper.
\end{example}

\subsection{The multisymplectic conservation law} Define the
collection of \emph{canonical 2-forms}
$ \omega ^\mu \coloneqq \mathrm{d} u ^i \wedge \mathrm{d} \sigma _i
^\mu $ on $ \mathbb{R}^n \times \mathbb{R} ^{ m n } $, for
$ \mu = 1 , \ldots, m $.

\begin{notation}
  Unless otherwise stated, differential forms and exterior
  differential operators (such as $ \mathrm{d} $, $ \wedge $, etc.)
  are assumed to be on $ \mathbb{R} ^n \times \mathbb{R} ^{ m n } $,
  where $ x \in U $ (if it appears) is fixed. Differentiation with
  respect to $x$ will always be denoted using the previously-defined
  $ \partial _\mu $ notation.
\end{notation}

\begin{definition}
  Let $ ( u, \sigma ) $ be a solution to \eqref{eqn:pdeSystem}. A
  \emph{(first) variation} of $ ( u , \sigma ) $ is a solution
  $ ( v, \tau ) $ of the linearized problem
  \begin{subequations}
    \label{eqn:linearizedProblem}
    \begin{align}
      \partial _\mu v ^i &= \frac{ \partial \phi ^i _\mu }{ \partial u ^j } ( \cdot , u , \sigma ) v ^j + \frac{ \partial \phi ^i _\mu }{ \partial \sigma _j ^\nu } ( \cdot , u , \sigma ) \tau _j ^\nu ,\\
      - \partial _\mu \tau  _i ^\mu &= \frac{ \partial f _i }{ \partial u ^j } ( \cdot , u , \sigma ) v ^j + \frac{ \partial f _i }{ \partial \sigma _j ^\nu } ( \cdot , u , \sigma ) \tau _j ^\nu .
    \end{align}
  \end{subequations}
  The system \eqref{eqn:pdeSystem} is \emph{multisymplectic} if
  $ \partial _\mu \Bigl( \omega ^\mu \bigl( ( v , \tau ) , ( v ^\prime
  , \tau ^\prime ) \bigr) \Bigr) = 0 $ for any pair of variations
  $ ( v, \tau ) $ and $ ( v ^\prime , \tau ^\prime ) $. This is
  abbreviated by
  \begin{equation}
    \label{eqn:msclDifferential}
    \partial _\mu \omega ^\mu = 0 ,
  \end{equation}
  where it is understood that $ \omega ^\mu $ is evaluated on
  variations of solutions to \eqref{eqn:pdeSystem}.  The equation
  \eqref{eqn:msclDifferential} is called the \emph{multisymplectic
    conservation law}.
\end{definition}

\begin{lemma}
  \label{lem:multisymplecticClosed}
  The system \eqref{eqn:pdeSystem} is multisymplectic if and only if
  the $1$-form
  $ \phi ^i _\mu \,\mathrm{d}\sigma _i ^\mu + f _i \,\mathrm{d}u ^i $
  on $ \mathbb{R}^n \times \mathbb{R} ^{ m n } $ is closed for each
  $ x \in U $.
\end{lemma}

\begin{proof}
  Using \eqref{eqn:pdeSystem}, we calculate
  \begin{align*}
    \partial _\mu \omega ^\mu
    &= \mathrm{d} ( \partial _\mu u ^i ) \wedge \mathrm{d} \sigma _i ^\mu + \mathrm{d} u ^i \wedge \mathrm{d} ( \partial _\mu \sigma _i ^\mu ) \\
    &= \mathrm{d} \phi ^i _\mu \wedge \mathrm{d} \sigma _i ^\mu - \mathrm{d} u ^i \wedge \mathrm{d} f _i  \\
    &= \mathrm{d} ( \phi ^i _\mu \,\mathrm{d}\sigma _i ^\mu + f _i \,\mathrm{d}u ^i ) ,
  \end{align*}
  so the first expression vanishes if and only if the last expression
  vanishes.
\end{proof}

\begin{remark}
  Certain steps in this calculation, such as
  $ \partial _\mu (\mathrm{d} u ^i) = \mathrm{d} ( \partial _\mu u ^i
  ) $, are seen to be valid by evaluating both sides on arbitrary
  variations of $ ( u, \sigma ) $:
  \begin{equation*}
    \partial _\mu \bigl( \mathrm{d} u ^i ( v, \tau ) \bigr) = \partial _\mu v ^i = \mathrm{d} \phi ^i _\mu ( v, \tau ) = \mathrm{d} ( \partial _\mu u ) ( v, \tau ) .
  \end{equation*}
  Similar calculations with differential forms will appear throughout
  this paper, where they are interpreted as holding for arbitrary
  variations of solutions.
\end{remark}

\begin{corollary}
  \label{cor:hamImpliesMult}
  Every Hamiltonian system is multisymplectic.
\end{corollary}

\begin{proof}
  From \eqref{eqn:ddw}, we have
  \begin{equation*}
    \phi ^i _\mu \,\mathrm{d}\sigma _i ^\mu + f _i \,\mathrm{d}u ^i = \frac{ \partial H }{ \partial \sigma _i ^\mu } \,\mathrm{d}\sigma _i ^\mu + \frac{ \partial H }{ \partial u ^i } \,\mathrm{d}u ^i = \mathrm{d} H ,
  \end{equation*}
  which is exact and therefore closed.
\end{proof}

\begin{corollary}
  \label{cor:multImpliesHam}
  Every multisymplectic system is Hamiltonian.
\end{corollary}

\begin{proof}
  Since $ \mathbb{R}^n \times \mathbb{R} ^{ m n } $ is simply
  connected, its first de~Rham cohomology is trivial. Hence, the
  closed $1$-form
  $ \phi ^i _\mu \,\mathrm{d}\sigma _i ^\mu + f _i \,\mathrm{d}u ^i $
  is exact, i.e., it equals $ \mathrm{d} H $ for some Hamiltonian
  $H$. More precisely, for each fixed $ x \in U $, the $1$-form equals
  $ \mathrm{d} H _x $ for some function $H _x $ on
  $ \mathbb{R}^n \times \mathbb{R} ^{ m n } $, and these may be
  combined into a single Hamiltonian
  $ H ( x, u, \sigma ) \coloneqq H _x ( u, \sigma ) $.
\end{proof}

\begin{remark}
  \autoref{cor:multImpliesHam} depends entirely on the fact that
  $ \mathbb{R}^n \times \mathbb{R} ^{ m n } $ has trivial first
  de~Rham cohomology. However, it is possible to define canonical
  Hamiltonian systems on more general spaces---in particular, on the
  dual jet bundle of some fiber bundle over $U$ (cf.~\citet{Gotay1991}
  and references therein). In this setting, the argument of
  \autoref{cor:multImpliesHam} holds only if the fibers of this bundle
  have trivial first de~Rham cohomology. However, a weaker
  statement---that every multisymplectic system is \emph{locally}
  Hamiltonian---still holds, by Poincar\'e's lemma.  By contrast,
  \autoref{cor:hamImpliesMult} remains true even in this more general
  setting.
\end{remark}

\begin{example}[semilinear elliptic PDE, continued]
  \label{ex:semilinearMSCLdiff}
  Let us see how the multisymplectic conservation law manifests in the
  class of semilinear elliptic PDEs we encountered in
  \autoref{ex:semilinearIntro}. For the system
  \begin{equation*}
    \partial _\mu u = a _{ \mu \nu } \sigma ^\nu , \qquad - \partial _\mu \sigma ^\mu = \frac{ \partial F }{ \partial u },
  \end{equation*}
  we calculate
  \begin{align*} 
    \partial _\mu \omega ^\mu
    &= \partial _\mu ( \mathrm{d} u \wedge \mathrm{d} \sigma ^\mu ) \\
    &= \mathrm{d} ( \partial _\mu u ) \wedge \mathrm{d} \sigma ^\mu + \mathrm{d} u \wedge \mathrm{d} ( \partial _\mu \sigma ^\mu ) \\
    &=  a _{ \mu \nu } \,\mathrm{d}\sigma ^\nu \wedge \mathrm{d} \sigma ^\mu + \mathrm{d} u \wedge \biggl( - \frac{  \partial ^2 F }{ \partial u ^2 } \,\mathrm{d}u - \frac{ \partial ^2 F }{ \partial \sigma \partial u } \,\mathrm{d}\sigma \biggr).
  \end{align*}
  The first term vanishes by the symmetry of $a$ and the antisymmetry
  of $ \wedge $, while the remaining terms vanish since
  $ \mathrm{d} u \wedge \mathrm{d} u = 0 $ (again, the antisymmetry of
  $ \wedge $) and since $F = F (x,u) $ does not depend on $\sigma$.
\end{example}

\subsection{Integral form of the multisymplectic conservation law}
\label{sec:msclIntegral}
Given an arbitrary subdomain
$ K \Subset U $, the divergence theorem implies that
\begin{equation*}
  \int _K \partial _\mu \omega ^\mu \,\mathrm{d} ^m x = \int _{ \partial K } \omega ^\mu \,\mathrm{d} ^{ m -1 } x _\mu ,
\end{equation*}
where
$ \mathrm{d} ^m x \coloneqq \mathrm{d} x ^1 \wedge \cdots \wedge
\mathrm{d} x ^m $ is the standard Euclidean volume form on $U$ and
$ \mathrm{d}^{m-1} x_\mu \coloneqq \iota _{ \partial / \partial x ^\mu
} ( \mathrm{d} ^m x ) $, where $ \iota $ is the interior product (or
contraction).  Therefore, an equivalent formulation of the
multisymplectic conservation law \eqref{eqn:msclDifferential} is that
\begin{equation}
  \label{eqn:msclIntegral}
  \int _{ \partial K } \omega ^\mu \,\mathrm{d}^{m-1} x_\mu = 0 , \quad \forall  K \Subset U.
\end{equation}
We call this the \emph{integral form of the multisymplectic
  conservation law}. As with \eqref{eqn:msclDifferential}, this is
interpreted as holding when $ \omega ^\mu $ is evaluated on arbitrary
variations of a solution to \eqref{eqn:pdeSystem}.

Note that, by the definition of the $ \wedge $ product,
\begin{equation*}
  \omega ^\mu = \mathrm{d} u ^i \wedge \mathrm{d} \sigma _i ^\mu = \mathrm{d} u ^i \otimes \mathrm{d} \sigma _i ^\mu - \mathrm{d} \sigma _i ^\mu \otimes \mathrm{d} u ^i ,
\end{equation*}
so \eqref{eqn:msclIntegral} may also be written as
\begin{equation}
  \label{eqn:msclTensor}
  \int _{ \partial K } ( \mathrm{d} u ^i \otimes \mathrm{d} \sigma _i ^\mu ) \,\mathrm{d}^{m-1} x_\mu = \int _{ \partial K } ( \mathrm{d} \sigma _i ^\mu \otimes \mathrm{d} u ^i ) \,\mathrm{d}^{m-1} x_\mu .
\end{equation}
Hence, the multisymplectic conservation law may be interpreted as a
symmetry condition on the Poincar\'e--Steklov operator mapping
Dirichlet boundary conditions for $u$ to the corresponding boundary
conditions for $\sigma$. See \citet{Agoshkov1988}, where the symmetry
of the Poincar\'e--Steklov operator is discussed in the context of
domain decomposition methods for linear elliptic problems.  Similarly,
\citet{BeSh2008} establish the symmetry of a Poincar\'e--Steklov
operator for harmonic differential forms on a manifold with boundary.

\begin{example}[semilinear elliptic PDE, continued]
  \label{ex:semilinearMSCLint}
  Let us revisit the class of semilinear elliptic PDEs we encountered
  in \autoref{ex:semilinearIntro} and
  \autoref{ex:semilinearMSCLdiff}. Let $ ( u, \sigma ) $ be a
  solution, and consider the linearized problem
  \begin{equation*}
    \partial _\mu v = a _{ \mu \nu } \tau ^\nu , \qquad - \partial _\mu \tau ^\mu = \frac{ \partial ^2 F }{ \partial u ^2 } v .
  \end{equation*}
  Here, $ \partial ^2 F / \partial u ^2 $ is evaluated at
  $ \bigl( x, u (x) \bigr) $ and hence is a function of $x$ alone.  If
  $ ( v, \tau ) $ and $ ( v ^\prime , \tau ^\prime ) $ are two
  arbitrary solutions to this problem, then
  \begin{align*}
    \int _{ \partial K } v \tau ^{\prime \mu} \,\mathrm{d}^{m-1} x_\mu &= \int _K \bigl[  (\partial _\mu v) \tau ^{ \prime \mu }  + v (\partial _\mu \tau ^{ \prime \mu } ) \bigr]  \,\mathrm{d}^m x \\
                                                                       &= \int _K \biggl( a _{ \mu \nu } \tau ^\nu \tau ^{ \prime \mu } - v \frac{ \partial ^2 F }{ \partial u ^2 } v ^\prime \biggr) \,\mathrm{d}^m x .
  \end{align*}
  By a similar calculation, switching $ ( v, \tau ) $ with
  $ ( v ^\prime, \tau ^\prime ) $,
  \begin{equation*}
    \int _{ \partial K } v ^\prime \tau ^\mu \,\mathrm{d}^{m-1} x_\mu = \int _K \biggl( a _{ \mu \nu } \tau ^{\prime \nu} \tau ^\mu - v ^\prime \frac{ \partial ^2 F }{ \partial u ^2 } v \biggr) \,\mathrm{d}^m x .
  \end{equation*}
  However, since $a$ is symmetric, these integrals are identical, and
  we conclude
  \begin{equation*}
    \int _{ \partial K } v \tau ^{\prime \mu} \,\mathrm{d}^{m-1} x_\mu =     \int _{ \partial K } v ^\prime \tau ^\mu \,\mathrm{d}^{m-1} x_\mu .
  \end{equation*}
  That this equality holds for every $ ( v, \tau ) $ and
  $ ( v ^\prime, \tau ^\prime ) $ is precisely the statement
  \eqref{eqn:msclTensor} of the multisymplectic conservation law.

  In the special case $ a ^{ \mu \nu } \equiv \delta ^{ \mu \nu } $,
  we have $ \tau = \operatorname{grad} v $ and
  $ \tau ^\prime = \operatorname{grad} v ^\prime $, so this can be
  written as
  \begin{equation*}
    \int _{ \partial K } v \operatorname{grad} v ^\prime \cdot \mathbf{n} = \int _{ \partial K } v ^\prime \operatorname{grad} v \cdot \mathbf{n} 
  \end{equation*} 
  where $ \mathbf{n} $ denotes the outer unit normal to
  $ \partial K $. Hence, in this case, the multisymplectic
  conservation law expresses the symmetry of the Dirichlet--Neumann
  operator
  $ v \rvert _{ \partial K } \mapsto \operatorname{grad} v \cdot
  \mathbf{n} \rvert _{ \partial K } $ for the linearized problem.
\end{example}

\subsection{Multisymplecticity and reciprocity}

In many physical systems, the multisymplectic conservation law is
closely tied to so-called \emph{reciprocity} phenomena, such as
Green's reciprocity in electrostatics and Betti reciprocity in
elasticity. (See, for example, \citet[Section 5.3]{AbMa1978},
\citet[Section 5.6]{MaHu1994}, and \citet{LeMaOrWe2003}.) These
reciprocity phenomena are also exploited, numerically, in formulations
of the boundary element method (cf.~\citet{PaBrWr1992}). We now
briefly discuss the relationship between multisymplecticity and
reciprocity, using the language we have developed throughout this
section.

Let $ ( u, \sigma ) $ be a solution to \eqref{eqn:pdeSystem}.  The
multisymplectic conservation law \eqref{eqn:msclDifferential} is just
the statement that
\begin{equation*}
  \partial _\mu ( v ^i \tau _i ^{ \prime \mu } ) = \partial _\mu ( v ^{\prime i} \tau _i ^\mu ) ,
\end{equation*}
where $ ( v, \tau ) $ and $ ( v ^\prime , \tau ^\prime ) $ are
arbitrary variations of $ ( u, \sigma ) $, i.e., solutions to the
linearized problem \eqref{eqn:linearizedProblem}. Integrating both
sides over $ K \Subset U $ and applying the divergence theorem gives
\begin{equation*}
  \int _{ \partial K } v ^i \tau _i ^{ \prime \mu } \,\mathrm{d}^{m-1} x_\mu =   \int _{ \partial K } v ^{\prime i} \tau _i ^\mu \,\mathrm{d}^{m-1} x_\mu,
\end{equation*}
which is the statement of \eqref{eqn:msclIntegral} and
\eqref{eqn:msclTensor}.

We now generalize the above to the case where the variations
$ ( v, \tau ) $ and $ ( v ^\prime , \tau ^\prime ) $ each solve a
\emph{perturbed} version of the linearized problem.

\begin{definition}
  Given a solution $ ( u, \sigma ) $ to \eqref{eqn:pdeSystem}, we say
  that $ ( v, \tau ) $ solves the \emph{linearized problem with
    incremental sources} $\psi$ and $g$ if
  \begin{subequations}
    \label{eqn:linearized}
    \begin{align}
      \partial _\mu v ^i &= \frac{ \partial \phi ^i _\mu }{ \partial u ^j } ( \cdot , u , \sigma ) v ^j + \frac{ \partial \phi ^i _\mu }{ \partial \sigma _j ^\nu } ( \cdot , u , \sigma ) \tau _j ^\nu + \psi ^i _\mu ( \cdot , v, \tau ) ,       \label{eqn:linearizedA}\\
      - \partial _\mu \tau  _i ^\mu &= \frac{ \partial f _i }{ \partial u ^j } ( \cdot , u , \sigma ) v ^j + \frac{ \partial f _i }{ \partial \sigma _j ^\nu } ( \cdot , u , \sigma ) \tau _j ^\nu + g _i ( \cdot , v , \tau )       \label{eqn:linearizedB}      ,
  \end{align}
\end{subequations}
  where $ \psi = \psi ^i _\mu (x, v, \tau ) $ and
  $ g = g _i ( x, v, \tau ) $ are given functions.
\end{definition}

Let $ ( v, \tau ) $ be a solution to the linearized problem with
incremental sources $ \psi $ and $g$, and let
$ ( v ^\prime , \tau ^\prime ) $ be a solution to the linearized
problem with incremental sources $ \psi ^\prime $ and $g ^\prime $. By
the Leibniz rule,
\begin{align*}
  \partial _\mu ( v ^i \tau _i ^{ \prime \mu } ) &= ( \partial _\mu v ^i ) \tau _i ^{ \prime \mu } + v ^i (\partial _\mu \tau _i ^{ \prime \mu }) ,\\
  \partial _\mu ( v ^{\prime i} \tau _i ^\mu ) &= ( \partial _\mu v ^{\prime i} ) \tau _i ^\mu + v ^{\prime i} (\partial _\mu \tau _i ^\mu) ,
\end{align*}
which we subtract and rearrange to obtain
\begin{equation*}
  \partial _\mu ( v ^i \tau _i ^{ \prime \mu } ) - ( \partial _\mu v ^i ) \tau _i ^{ \prime \mu } - v ^i (\partial _\mu \tau _i ^{ \prime \mu }) = \partial _\mu ( v ^{\prime i} \tau _i ^\mu ) - ( \partial _\mu v ^{\prime i} ) \tau _i ^\mu - v ^{\prime i} (\partial _\mu \tau _i ^\mu) .
\end{equation*}
Assuming that \eqref{eqn:pdeSystem} is multisymplectic, the $\phi$ and
$f$ terms cancel when we substitute \eqref{eqn:linearized}, leaving
the equation
\begin{multline}
  \label{eqn:recipDiff}
  \partial _\mu ( v ^i \tau _i ^{ \prime \mu } ) - \psi ^i _\mu ( \cdot , v, \tau ) \tau _i ^{ \prime \mu } + v ^i g _i ^\prime ( \cdot , v ^\prime , \tau ^\prime ) \\
  = \partial _\mu ( v ^{\prime i} \tau _i ^\mu ) - \psi ^{\prime i}
  _\mu ( \cdot , v ^\prime , \tau ^\prime ) \tau _i ^\mu + v ^{\prime
    i} g _i ( \cdot , v , \tau )
\end{multline}
Integrating over $ K \Subset U $ and applying the divergence theorem gives
\begin{multline}
  \label{eqn:recipInt}
  \int _{ \partial K } v ^i \tau _i ^{ \prime \mu } \,\mathrm{d}^{m-1} x_\mu  - \int _K \bigl[ \psi ^i _\mu ( \cdot , v, \tau ) \tau _i ^{ \prime \mu } - v ^i g _i ^\prime ( \cdot , v ^\prime , \tau ^\prime ) \bigr] \,\mathrm{d}^m x \\
  = \int _{ \partial K }  v ^{\prime i} \tau _i ^\mu \,\mathrm{d}^{m-1} x_\mu  - \int _K \bigl[ \psi ^{\prime i}
  _\mu ( \cdot , v ^\prime , \tau ^\prime ) \tau _i ^\mu + v ^{\prime
    i} g _i ( \cdot , v , \tau )\bigr] \,\mathrm{d}^m x .
\end{multline}
The equations \eqref{eqn:recipDiff} and \eqref{eqn:recipInt} are the
differential and integral forms, respectively, of the reciprocity law
for a multisymplectic system of PDEs. They may be interpreted as
describing a symmetric (or ``reciprocal'') relationship between the
perturbation of the system by incremental sources and the incremental
response of the system to such perturbations. In the special case
where the incremental sources vanish, we recover the multisymplectic
conservation law.

\begin{example}[semilinear elliptic PDE, continued]
  Let us once again examine the semilinear elliptic PDEs considered in
  \autoref{ex:semilinearIntro}, \autoref{ex:semilinearMSCLdiff}, and
  \autoref{ex:semilinearMSCLint}. The linearized problem with
  incremental sources is
  \begin{equation*}
    \partial _\mu v = a _{ \mu \nu } \tau ^\nu + \psi _\mu (\cdot, v, \tau ) , \qquad - \partial _\mu \tau ^\mu = \frac{ \partial ^2 F }{ \partial u ^2 } v + g (\cdot , v , \tau ) .
  \end{equation*}
  To see how reciprocity arises, we compute
  \begin{align*}
    \partial _\mu ( v  \tau ^{ \prime \mu } ) &= \bigl[ a _{ \mu \nu } \tau ^\nu + \psi _\mu  (\cdot , v, \tau ) \bigr]   \tau ^{ \prime \mu } - v \biggl[ \frac{ \partial ^2 F }{ \partial u ^2 } v ^\prime + g ^\prime ( \cdot , v ^\prime , \tau ^\prime ) \biggr],\\
    \partial _\mu ( v ^\prime \tau ^\mu ) &= \bigl[ a _{ \mu \nu } \tau ^{\prime \nu} + \psi ^\prime _\mu  (\cdot , v ^\prime , \tau ^\prime ) \bigr]   \tau ^\mu - v ^\prime \biggl[ \frac{ \partial ^2 F }{ \partial u ^2 } v + g ( \cdot , v , \tau ) \biggr] .
  \end{align*}
  Subtracting, the terms involving $a$ and $F$ cancel by symmetry,
  yielding
  \begin{multline*}
    \partial _\mu ( v \tau ^{ \prime \mu } ) - \psi _\mu ( \cdot , v, \tau ) \tau ^{ \prime \mu } + v g ^\prime ( \cdot , v ^\prime , \tau ^\prime ) \\
    = \partial _\mu ( v ^\prime \tau ^\mu ) - \psi _\mu ^\prime (
    \cdot , v ^\prime , \tau ^\prime ) \tau ^\mu + v ^\prime g ( \cdot
    , v , \tau ),
  \end{multline*}
  which is precisely the statement \eqref{eqn:recipDiff}. Integrating
  this over $ K \Subset U $ gives
  \begin{multline*}
    \int _{ \partial K } v \tau ^{ \prime \mu } \,\mathrm{d}^{m-1} x_\mu  - \int _{ \partial K } \bigl[ \psi _\mu ( \cdot , v, \tau ) \tau ^{ \prime \mu } - v g ^\prime ( \cdot , v ^\prime , \tau ^\prime ) \bigr] \,\mathrm{d}^m x \\
    = \int _{ \partial K } v ^\prime \tau ^\mu \,\mathrm{d}^{m-1}
    x_\mu - \int _K \bigl[ \psi _\mu ^\prime ( \cdot , v ^\prime ,
    \tau ^\prime ) \tau ^\mu - v ^\prime g ( \cdot , v , \tau ) \bigr]
    \,\mathrm{d}^{m-1} x_\mu ,
  \end{multline*}
  which is the statement \eqref{eqn:recipInt}.

  As an important special case, which arises in the primal (or
  Lagrangian) formulation of this system, suppose that $ \psi = 0 $
  and that $ g = g ( \cdot , v ) $, so that
  $ \tau = a \operatorname{grad} v $ and $v$ solves the second-order
  equation
  \begin{equation*}
    - \operatorname{div} ( a \operatorname{grad} v ) = \frac{ \partial ^2 F }{ \partial u ^2 } v + g ( \cdot , v ) .
  \end{equation*}
  If the corresponding properties hold for $ ( v ^\prime , \tau ^\prime ) $,
  then we can write the reciprocity law as 
  \begin{equation*}
    \operatorname{div} ( v \tau ^\prime ) + v g ^\prime ( \cdot , v ^\prime ) = \operatorname{div} ( v ^\prime \tau ) + v ^\prime g ( \cdot , v ) ,
  \end{equation*}
  whose integral form on $ K \Subset U $ is
  \begin{equation*}
    \int _{ \partial K } v \tau ^\prime \cdot \mathbf{n} + \int _K v g ^\prime ( \cdot , v ^\prime ) = \int _{ \partial K } v ^\prime \tau \cdot \mathbf{n} + \int _K v ^\prime g ( \cdot , v ) ,
  \end{equation*}
  As a final specialization, let
  $ a ^{ \mu \nu } \equiv \delta ^{ \mu \nu } $ and
  $ F ( x, u ) = f (x) u $, so that $u$ satisfies Poisson's equation,
  $ - \Delta u = f $, and $v$ and $ v ^\prime $ satisfy
  \begin{equation*}
    - \Delta v = g ( \cdot , v ) , \qquad - \Delta v ^\prime = g ^\prime ( \cdot , v ^\prime ) .
  \end{equation*}
  Then the reciprocity law, in differential form, is
  \begin{equation*}
    \operatorname{div} ( v \tau ^\prime ) - v \Delta v ^\prime = \operatorname{div} ( v ^\prime \tau ) - v ^\prime \Delta v ,
  \end{equation*}
  while the integral form on $ K \Subset U $ is
  \begin{equation*}
    \int _{ \partial K } v \tau ^\prime \cdot \mathbf{n} - \int _K v \Delta v ^\prime = \int _{ \partial K } v ^\prime \tau \cdot \mathbf{n} - \int _K v ^\prime \Delta v .
  \end{equation*}
  These last two expressions are two of Green's identities from vector
  calculus. If $v$ and $ v ^\prime $ are interpreted as scalar
  potentials for the electrostatic fields $ \tau $ and
  $ \tau ^\prime $, respectively, then this corresponds to Green's
  reciprocity.
\end{example}

\section{The flux formulation and multisymplecticity}
\label{sec:flux}

\subsection{Domain decomposition and the flux formulation}
\label{sec:fluxFormulation}

In this section, we introduce a weak formulation of the problem
\eqref{eqn:pdeSystem}, called the \emph{flux formulation}. This
decomposes the problem on $U$ into a collection of local solvers for
$ ( u, \sigma ) $, coupled through the approximate boundary traces
$ ( \widehat{ u }, \widehat{ \sigma } ) $. This forms the foundation
of the HDG framework of \citet{CoGoLa2009} and is closely related to
the unified {DG} framework of \citet{ArBrCoMa2001}.

We mention that our presentation of the flux formulation, and of HDG
methods, differs from that in \citet{CoGoLa2009} in a few ways. In
particular, \citeauthor{CoGoLa2009} focus on linear elliptic problems,
which allows them to make substantial use of the solution theory for
such problems, including well-posedness of the local and global
solvers. By contrast, we are interested in obtaining
multisymplecticity criteria for the much more general class of systems
\eqref{eqn:pdeSystem}, without assuming anything about the properties
of solutions, even their existence and/or uniqueness.

In this section, the function spaces appearing in the flux formulation
may be either infinite-dimensional Hilbert spaces or
finite-dimensional subspaces (e.g., polynomials up to some
degree). This will allow us to prove general multisymplecticity
results that apply both to the original, infinite-dimensional problem
(\autoref{sec:exactSolutions}) and to finite-element approximation via
HDG methods (\autoref{sec:hdg}).

To begin, observe that if $ ( u, \sigma ) $ is a smooth solution to
\eqref{eqn:pdeSystem} on $U$, and if $ ( v, \tau ) $ are arbitrary
smooth test functions, then
\begin{align*}
  \int _U \partial _\mu u ^i \tau _i ^\mu \,\mathrm{d}^m x &= \int _U \phi ^i _\mu \tau _i ^\mu \,\mathrm{d}^m x, \\
  - \int _U \partial _\mu \sigma _i ^\mu v ^i \,\mathrm{d}^m x &= \int _U f _i v ^i \,\mathrm{d}^m x .
\end{align*}
If $ \mathcal{T} _h $ is a partition of $U$ into non-overlapping
domains $ K \in \mathcal{T} _h $, then breaking each of the integrals
above into a sum over $ K \in \mathcal{T} _h $ and integrating by
parts,
\begin{subequations}
  \label{eqn:globalWeakProblem}
\begin{align}
  \sum _{K \in \mathcal{T} _h } \int _{ \partial K } u ^i \tau _i ^\mu \,\mathrm{d}^{m-1} x_\mu  &= \sum _{K \in \mathcal{T} _h } \int _K  ( u ^i \partial _\mu \tau _i ^\mu + \phi ^i _\mu \tau _i ^\mu ) \,\mathrm{d}^m x ,\\
  \sum _{K \in \mathcal{T} _h }  \int _{ \partial K } \sigma _i ^\mu v ^i  \,\mathrm{d}^{m-1} x_\mu  &= \sum _{K \in \mathcal{T} _h } \int _K ( \sigma _i ^\mu \partial _\mu v ^i - f _i v ^i  ) \,\mathrm{d}^m x .
\end{align}
\end{subequations}
In a typical finite-element application, $U$ will be polyhedral, and
$ \mathcal{T} _h $ will be a triangulation of $U$ into simplices
$ K \in \mathcal{T} _h $.

Following \citet{CoGoLa2009} (as well as \citet{ArBrCoMa2001}), we now
relax the regularity and inter-element continuity assumptions on
$ u , v $ and $ \sigma, \tau $ in \eqref{eqn:globalWeakProblem}, and
we replace the boundary traces of $u$ and $ \sigma $ on $ \partial K $
by approximate traces $ \widehat{ u } $ and $ \widehat{ \sigma }
$. Specifically, let
\begin{equation*}
  V (K) \subset \bigl[ H ^2 (K) \bigr] ^n , \qquad \Sigma (K) \subset \bigl[ H ^1 (K) \bigr] ^{ m n } ,
\end{equation*}
be specified local function spaces on each $K \in \mathcal{T} _h $,
and define discontinuous global spaces on $U$ by
\begin{alignat*}{2}
  V  &\coloneqq \bigl\{ v \in \bigl[ L ^2 (U) \bigr] ^n &: v \rvert _K \in V (K) ,\ \forall K \in \mathcal{T} _h \bigr\} &= \prod _{ K \in \mathcal{T} _h } V (K) , \\
  \Sigma &\coloneqq \bigl\{ \tau  \in \bigl[ L ^2 (U) \bigr] ^{m n} &: \tau  \rvert _K \in \Sigma  (K) , \ \forall  K \in \mathcal{T} _h \bigr\} &= \prod _{ K \in \mathcal{T} _h } \Sigma  (K).
\end{alignat*}
Next, specify a space of approximate traces of functions in $V$,
\begin{equation*}
  \widehat{ V } \subset \bigl[ L ^2 ( \mathcal{E} _h ) \bigr] ^n ,
\end{equation*}
where
$ \mathcal{E} _h \coloneqq \bigcup _{ K \in \mathcal{T} _h } \partial
K $, along with the subspace
\begin{equation*}
  \widehat{ V } _0 \coloneqq \bigl\{ \widehat{ v } \in \widehat{ V } : \widehat{ v } \rvert _{ \partial U } = 0 \bigr\}
\end{equation*} 
of approximate traces vanishing on the domain boundary $ \partial U $.

The final ingredient in the HDG framework is the \emph{numerical flux}
$ \widehat{ \sigma } $, which we define slightly differently to
\citet{CoGoLa2009}. As mentioned in the introduction to this section,
our treatment is equivalent to \citep{CoGoLa2009} for linear problems,
but it extends more naturally to nonlinear problems, even without
assuming existence and uniqueness of solutions. Let
\begin{equation*}
  \widehat{ \Sigma } ( \partial K ) \subset \bigl[ L ^2 (\partial K )
  \bigr] ^{ m n }
\end{equation*}
be some space of boundary fluxes on $ \partial K $, and define the
space of restricted traces
$ \widehat{ V } ( \partial K ) \coloneqq \bigl\{ \widehat{ v } \rvert
_{ \partial K } : \widehat{ v } \in \widehat{ V } \bigr\} $. 

\begin{definition}
  A \emph{local flux function} on $ K \in \mathcal{T} _h $ is a bounded
  linear map
  \begin{equation*}
    \Phi _K  \colon V (K) \times \Sigma (K) \times \widehat{ V } ( \partial K ) \times \widehat{ \Sigma } ( \partial K ) \rightarrow \bigl[ L ^2 ( \partial K ) \bigr] ^{ m n } .
  \end{equation*}
  Denoting
  $ \widehat{ \Sigma } \coloneqq \prod _{ K \in \mathcal{T} _h }
  \widehat{ \Sigma } ( \partial K ) \subset \prod _{ K \in \mathcal{T}
    _h } \bigl[ L ^2 ( \partial K ) \bigr] ^{ m n } $, this extends
  naturally to a \emph{global flux function},
  \begin{equation*}
    \Phi \colon V \times \Sigma \times \widehat{ V } \times \widehat{ \Sigma } \rightarrow \prod _{ K \in \mathcal{T} _h } \bigl[ L ^2 ( \partial K ) \bigr] ^{ m n } .
  \end{equation*} 
\end{definition}

\begin{remark}
  \label{rmk:doubleValued}
  Elements of $ \widehat{ \Sigma } $ and
  $ \prod _{ K \in \mathcal{T} _h } \bigl[ L ^2 ( \partial K ) \bigr]
  ^{ m n } $ may be interpreted as functions that are double-valued on
  internal facets of $ \mathcal{T} _h $ and single-valued on boundary
  facets in $ \partial U $.
\end{remark}

We now seek solutions
$ ( u, \sigma , \widehat{ u } , \widehat{ \sigma } ) \in V \times
\Sigma \times \widehat{ V } \times \widehat{ \Sigma } $ satisfying
\begin{subequations}
\label{eqn:fluxHDG}  
\begin{alignat}{2}
  \label{eqn:localSolverA}
  \int _{ \partial K } \widehat{ u } ^i \tau _i ^\mu \,\mathrm{d}^{m-1} x_\mu &= \int _K ( u ^i \partial _\mu \tau _i ^\mu + \phi ^i _\mu \tau _i ^\mu ) \,\mathrm{d}^m x , \qquad & \forall \tau &\in \Sigma (K) ,\\
  \label{eqn:localSolverB}
  \int _{ \partial K } \widehat{ \sigma } _i ^\mu v ^i \,\mathrm{d}^{m-1} x_\mu &= \int _K ( \sigma _i ^\mu \partial _\mu v ^i - f _i v ^i ) \,\mathrm{d}^m x , \qquad &\forall v &\in V (K) ,\\
  \label{eqn:numericalFlux}
  \mathclap{\int _{ \partial K } \Phi ^i _\mu  ( u , \sigma , \widehat{ u } , \widehat{ \sigma } )  \widehat{ \tau } _i ^\mu \mathbf{n} ^\nu \,\mathrm{d}^{m-1} x_\nu = 0,}  &&\forall \widehat{ \tau } &\in \widehat{ \Sigma } ( \partial K ) ,
\end{alignat}
for all $ K \in \mathcal{T} _h $, together with the
\emph{conservativity condition},
\begin{equation}
  \label{eqn:conservativity}
  \sum _{ K \in \mathcal{T} _h } \int _{ \partial K } \widehat{ \sigma } _i ^\mu \widehat{ v } ^i \,\mathrm{d}^{m-1} x_\mu = 0 , \qquad \forall \widehat{ v } \in \widehat{ V } _0 ,
\end{equation}
\end{subequations}
the latter of which states that the normal component of
$ \widehat{ \sigma } $ is single-valued (at least in a weak sense) on
the internal facets of $ \mathcal{T} _h $.

In the language of \citet{CoGoLa2009}, we have ``local solvers''
\eqref{eqn:localSolverA}--\eqref{eqn:localSolverB} on each
$ K \in \mathcal{T} _h $, and these are coupled globally through the
numerical flux $ \widehat{ \sigma } $ by the conservativity condition
\eqref{eqn:conservativity}. A notable distinction between our approach
and that of \citet{CoGoLa2009} is that they assume
$ \widehat{ \sigma } = \widehat{ \sigma } ( u, \sigma , \widehat{ u }
) $ is a given function of $ u $, $ \sigma $, and $ \widehat{ u } $ on
each $ K \in \mathcal{T} _h $, whereas we define it through the flux
functions $ \Phi _K $ by adding \eqref{eqn:numericalFlux} to the flux
formulation.

\begin{definition}
  The \emph{flux formulation} of \eqref{eqn:pdeSystem} on
  $ \mathcal{T} _h $ is given by \eqref{eqn:fluxHDG}, along with
  choices of the global function space $ \widehat{ V } $ and, for each
  $ K \in \mathcal{T} _h $, the local function spaces $ V (K) $,
  $ \Sigma (K) $, $ \widehat{ \Sigma } ( \partial K ) $ and the flux
  function $ \Phi _K $. We call this a \emph{hybridizable
    discontinuous Galerkin (HDG) method} whenever
  $ \widehat{ V } _0 $, $V$, and $\Sigma$ (but not necssarily
  $ \widehat{ V } $ or $ \widehat{ \Sigma } $) are finite dimensional.
\end{definition}

\begin{remark}
  Note that \eqref{eqn:fluxHDG} does not impose any particular
  boundary conditions on $ \widehat{ u } \rvert _{ \partial U } $ or
  $ \widehat{ \sigma } \rvert _{ \partial U } $. Hence,
  \eqref{eqn:fluxHDG} corresponds to the system of PDEs
  \eqref{eqn:pdeSystem} rather than a particular boundary value
  problem associated to \eqref{eqn:pdeSystem}.

  We remain agnostic about the choice of boundary conditions for two
  reasons.  First, multisymplecticity is not a statement about a
  particular solution, but a statement about variations within a
  general family of solutions. If we pick out an isolated solution
  (e.g., by the imposition of boundary conditions), then the
  ``family'' of solutions becomes zero-dimensional, so any statement
  about variations is vacuous.  Second, the class of PDEs
  \eqref{eqn:pdeSystem} is quite general, including both elliptic and
  hyperbolic PDEs, among others, depending on $\phi$ and $f$. In the
  hyperbolic case, when $U$ is a spacetime region, we are not free to
  impose Dirichlet conditions on all of $ \partial U $.
\end{remark}

\subsection{Local multisymplecticity criteria}
In the context of smooth solutions to \eqref{eqn:pdeSystem}, where
$ \phi ^i _\mu $ and $f _i $ are smooth functions on
$ U \times \mathbb{R}^n \times \mathbb{R} ^{ m n } $,
\autoref{lem:multisymplecticClosed} states that multisymplecticity
holds if and only if the smooth $1$-form
$ \phi ^i _\mu \,\mathrm{d}\sigma _i ^\mu + f _i \,\mathrm{d}u ^i $ is
closed for all $ x \in U $. For the flux formulation
\eqref{eqn:fluxHDG}, we wish to relax these smoothness assumptions and
express the multisymplecticity condition in terms of function spaces,
rather than in a pointwise sense at each $ x \in U $.

Observe that \eqref{eqn:localSolverA}--\eqref{eqn:localSolverB} still
makes sense even if we only have $ \phi ^i _\mu , f _i \in L ^2 (U) $
for $ \mu = 1 , \ldots, m $ and $ i = 1, \ldots, n $. Therefore,
rather than assuming that $\phi$ and $f$ are smooth, let us assume
only that
\begin{equation*}
  \phi _K \colon V (K) \times \Sigma (K) \rightarrow \bigl[ L ^2 (K) \bigr] ^{ m n } , \qquad f _K \colon V (K) \times \Sigma (K) \rightarrow \bigl[ L ^2 (K) \bigr] ^n ,
\end{equation*}
for each $ K \in \mathcal{T} _h $, which may be extended naturally to
\begin{equation*}
  \phi \colon V \times \Sigma \rightarrow \bigl[ L ^2 (U) \bigr] ^{ m n } , \qquad f \colon V \times \Sigma \rightarrow \bigl[ L ^2 (U) \bigr] ^n .
\end{equation*}
These are generally nonlinear maps---and since multisymplecticity is a
statement about first variations of solutions, let us assume also that
these maps are at least $ C ^1 $. When $ V \times \Sigma $ is
infinite-dimensional, we may interpret this as a variational
derivative (either the G\^ateaux or Fr\'echet derivative, which are
equivalent for $ C ^1 $ maps, cf.~\citet[Corollary 2.10]{AbMaRa1988});
in the finite-dimensional case, this is just ordinary continuous
differentiability.  With $ \phi $ and $f$ defined in this way, it
follows that
$ \phi ^i _\mu \,\mathrm{d}\sigma _i ^\mu + f _i \,\mathrm{d}u ^i $ is
a $ C ^1 $ differential $1$-form on $ V \times \Sigma $, and we say
that this $1$-form is closed if its exterior derivative vanishes as a
$2$-form on $ V \times \Sigma $.

\begin{definition}
  The flux formulation \eqref{eqn:fluxHDG} is \emph{multisymplectic}
  if solutions satisfy
  \begin{equation}
    \label{eqn:msclHDG}
    \int _{ \partial K } (\mathrm{d} \widehat{ u } ^i \wedge \mathrm{d} \widehat{ \sigma } _i ^\mu) \,\mathrm{d}^{m-1} x_\mu = 0 ,
  \end{equation}
  for all $ K \in \mathcal{T} _h $, whenever the $ C ^1 $ differential
  $1$-form
  $ \phi ^i _\mu \,\mathrm{d}\sigma _i ^\mu + f _i \,\mathrm{d}u ^i $
  is closed.
\end{definition}

\begin{remark}
  The condition \eqref{eqn:msclHDG} is essentially the integral form
  of the multisymplectic conservation law from
  \autoref{sec:msclIntegral}, where the approximate traces
  $ \widehat{ u } $ and $ \widehat{ \sigma } $ are used instead of the
  actual traces of $u$ and $ \sigma $. Note that \eqref{eqn:msclHDG}
  only needs to hold for $ K \in \mathcal{T} _h $, not for arbitrary
  subdomains $ K \Subset U $ as in \eqref{eqn:msclIntegral}.

  We say that \eqref{eqn:msclHDG} is a \emph{local} multisymplecticity
  condition because it is a statement only about the local solvers
  \eqref{eqn:localSolverA}--\eqref{eqn:localSolverB} and numerical
  flux condition \eqref{eqn:numericalFlux} for each
  $ K \in \mathcal{T} _h $. We reserve the \emph{global}
  question---whether the multisymplectic conservation law holds for
  arbitrary unions of elements of $ \mathcal{T} _h $---for the next
  section, where the conservativity condition
  \eqref{eqn:conservativity} will also come into play.
\end{remark}

To characterize the multisymplecticity of the flux formulation, we
first prove a useful lemma, which relates the boundary integral in
\eqref{eqn:msclHDG} to the ``jumps'' $ \widehat{ u } - u $ and
$ \widehat{ \sigma } - \sigma $ between the approximate and actual
traces on $ \partial K $.

\begin{lemma}
  \label{lem:jump}
  If
  $ \phi ^i _\mu \,\mathrm{d}\sigma _i ^\mu + f _i \,\mathrm{d}u ^i $
  is closed and
  $ ( u, \sigma, \widehat{ u } , \widehat{ \sigma } ) \in V \times
  \Sigma \times \widehat{ V } \times \widehat{ \Sigma } $ satisfies
  \eqref{eqn:localSolverA}--\eqref{eqn:localSolverB} for
  $ K \in \mathcal{T} _h $, then
  \begin{equation*}
    \int _{ \partial K } (\mathrm{d} \widehat{ u } ^i \wedge \mathrm{d} \widehat{ \sigma } _i ^\mu )\,\mathrm{d}^{m-1} x_\mu = \int _{ \partial K } \bigl[ \mathrm{d} ( \widehat{ u } ^i - u ^i ) \wedge \mathrm{d} ( \widehat{ \sigma } _i ^\mu - \sigma _i ^\mu ) \bigr] \,\mathrm{d}^{m-1} x_\mu .
  \end{equation*}
  Consequently, the local multisymplecticity condition
  \eqref{eqn:msclHDG} holds if and only if
  \begin{equation}
    \label{eqn:msclJump}
    \int _{ \partial K } \bigl[ \mathrm{d} ( \widehat{ u } ^i - u ^i ) \wedge \mathrm{d} ( \widehat{ \sigma } _i ^\mu - \sigma _i ^\mu ) \bigr] \,\mathrm{d}^{m-1} x_\mu = 0 .
  \end{equation} 
\end{lemma}

\begin{proof}
  Since \eqref{eqn:localSolverA} and \eqref{eqn:localSolverB} hold for
  all $ v \in V (K) $ and $ \tau \in \Sigma (K) $, we have
  \begin{align*}
    \int _{ \partial K } (\widehat{ u } ^i \,\mathrm{d} \sigma  _i ^\mu) \,\mathrm{d}^{m-1} x_\mu &= \int _K \bigl[  u ^i \,\mathrm{d} ( \partial _\mu \sigma ^\mu _i) + \phi ^i _\mu \,\mathrm{d} \sigma  _i ^\mu \bigr]  \,\mathrm{d}^m x, \\
    \int _{ \partial K } ( \widehat{ \sigma } _i ^\mu \,\mathrm{d} u ^i )
    \,\mathrm{d}^{m-1} x_\mu &= \int _K \bigl[  \sigma _i ^\mu \,\mathrm{d} (\partial _\mu
 u ^i ) - f _i \,\mathrm{d} u ^i ) \,\mathrm{d}^m x ,
  \end{align*}
  so taking exterior derivatives gives
  \begin{align*}
    \int _{ \partial K } (\mathrm{d} \widehat{ u } ^i \wedge \mathrm{d} \sigma  _i ^\mu) \,\mathrm{d}^{m-1} x_\mu &= \int _K \bigl[  \mathrm{d} u ^i \wedge \mathrm{d} (\partial _\mu \sigma ^\mu _i) + \mathrm{d} \phi ^i _\mu \wedge \mathrm{d} \sigma  _i ^\mu \bigr]  \,\mathrm{d}^m x, \\
    \int _{ \partial K } ( \mathrm{d} \widehat{ \sigma } _i ^\mu \wedge \mathrm{d} u ^i )
    \,\mathrm{d}^{m-1} x_\mu &= \int _K \bigl[  \mathrm{d} \sigma _i ^\mu \wedge \mathrm{d} (\partial _\mu u ^i ) - \mathrm{d} f _i \wedge \mathrm{d} u ^i ) \,\mathrm{d}^m x .
  \end{align*}
  Subtracting the second equation from the first, the terms involving
  $ \phi $ and $f$ vanish by the closedness assumption, so we are left
  with
\begin{multline*} 
  \int _{ \partial K } ( \mathrm{d} \widehat{ u } ^i \wedge \mathrm{d} \sigma _i ^\mu + \mathrm{d} u ^i \wedge \mathrm{d} \widehat{ \sigma } _i ^\mu ) \,\mathrm{d}^{m-1} x_\mu \\
 \begin{aligned}
  &= \int _K \bigl[ \mathrm{d} u ^i \wedge \mathrm{d} ( \partial _\mu \sigma _i ^\mu ) + \mathrm{d} ( \partial _\mu u ^i ) \wedge \mathrm{d} \sigma _i ^\mu \bigr] \,\mathrm{d}^m x \\
  &= \int _K \partial _\mu ( \mathrm{d} u ^i \wedge \mathrm{d} \sigma _i ^\mu ) \,\mathrm{d}^m x \\
  &= \int _{ \partial K } (\mathrm{d} u ^i \wedge \mathrm{d} \sigma _i ^\mu ) \,\mathrm{d}^{m-1} x_\mu ,
\end{aligned}
\end{multline*}
that is,
\begin{equation*}
\int _{ \partial K } ( \mathrm{d} \widehat{ u } ^i \wedge \mathrm{d} \sigma _i ^\mu + \mathrm{d} u ^i \wedge \mathrm{d} \widehat{ \sigma } _i ^\mu - \mathrm{d} u ^i \wedge \mathrm{d} \sigma _i ^\mu ) \,\mathrm{d}^{m-1} x_\mu = 0 .
\end{equation*} 
Using this identity, we finally calculate
\begin{multline*}
  \int _{ \partial K } \bigl[ \mathrm{d} ( \widehat{ u } ^i - u ^i ) \wedge \mathrm{d} ( \widehat{ \sigma } _i ^\mu - \sigma _i ^\mu  )  \bigr] \,\mathrm{d}^{m-1} x_\mu \\
  \begin{aligned}
    &=\int _{ \partial K } ( \mathrm{d} \widehat{ u } ^i \wedge \mathrm{d} \widehat{ \sigma } _i ^\mu - \mathrm{d} \widehat{ u } ^i \wedge \mathrm{d} \sigma _i ^\mu - \mathrm{d} u ^i \wedge \mathrm{d} \widehat{ \sigma } _i ^\mu + \mathrm{d} u ^i \wedge \mathrm{d} \sigma _i ^\mu ) \,\mathrm{d}^{m-1} x_\mu\\
    &=     \int _{ \partial K } (\mathrm{d} \widehat{ u } ^i \wedge \mathrm{d} \widehat{ \sigma } _i ^\mu )\,\mathrm{d}^{m-1} x_\mu ,
  \end{aligned}
\end{multline*}
which completes the proof.
\end{proof}

The equation \eqref{eqn:msclJump} says that the multisymplecticity of
the flux formulation depends entirely on the relationship among $ u $,
$\sigma$, $ \widehat{ u } $, and $ \widehat{ \sigma } $ on
$ \partial K $ for $K \in \mathcal{T} _h $. That is, it depends
entirely on the choice of local flux functions $ \Phi _K $.

\begin{definition}
  \label{def:multisymplecticFlux}
  A local flux function $ \Phi _K $ is \emph{multisymplectic} if
  \eqref{eqn:msclJump} holds whenever
  $ ( u, \sigma , \widehat{ u }, \widehat{ \sigma } ) \in V \times
  \Sigma \times \widehat{ V } \times \widehat{ \Sigma } $ satisfies
  \eqref{eqn:numericalFlux}.
\end{definition}

We now prove multisymplecticity for two particular choices of
$ \Phi _K $. The first is used for the hybridized Raviart--Thomas
(RT-H), Brezzi--Douglas--Marini (BDM-H), and local discontinuous
Galerkin (LDG-H) methods; the second is used for the hybridized
continuous Galerkin (CG-H) and nonconforming (NC-H) methods. These
methods will be discussed further in \autoref{sec:hdg}.

\begin{theorem}
  \label{thm:ldgFlux}
  Suppose that, for all $ v \in V (K) $ and
  $ \widehat{ v } \in \widehat{ V } $, there exists
  $ \widehat{ \tau } \in \widehat{ \Sigma } ( \partial K ) $ such that
  $ \widehat{ \tau } _i ^\mu \mathbf{n} _\mu = \delta _{ i j }
  (\widehat{ v } ^j - v ^j )\rvert _{ \partial K } $ for all
  $ i = 1 , \ldots, n $. Then, for any
  $ \lambda \in L ^\infty ( \partial K ) $ (which is called a
  ``penalty function''), the local flux function
  \begin{equation*}
    \Phi _K ( u , \sigma , \widehat{ u } , \widehat{ \sigma } ) = ( \widehat{ \sigma } - \sigma ) - \lambda ( \widehat{ u } - u ) \mathbf{n} 
  \end{equation*}
  is multisymplectic.
\end{theorem}

\begin{proof}
  The flux condition \eqref{eqn:numericalFlux} says that
  \begin{equation*}
    \int _{ \partial K } \delta ^{ i j } ( \widehat{ \sigma } _i ^\mu - \sigma _i ^\mu ) \mathbf{n} _\mu \widehat{ \tau } _j ^\nu \,\mathrm{d}^{m-1} x_\nu = \int _{ \partial K } \lambda ( \widehat{ u } ^i - u ^i ) \widehat{ \tau } _i ^\mu \,\mathrm{d}^{m-1} x_\mu ,
  \end{equation*}
  for all $ \widehat{ \tau } \in \widehat{ \Sigma } ( \partial K ) $.
  By assumption, for any $ v \in V (K) $ and
  $ \widehat{ v } \in \widehat{ V } $, we can choose
  $ \widehat{ \tau } \in \widehat{ \Sigma } ( \partial K ) $ such that
  $ \widehat{ \tau } _i ^\mu \mathbf{n} _\mu = \delta _{ i j }
  (\widehat{ v } ^j - v ^j ) \rvert _{ \partial K } $ for all
  $ i = 1 , \ldots, n $, and therefore
  \begin{equation*}
    \int _{ \partial K } ( \widehat{ \sigma } _i ^\mu - \sigma _i ^\mu ) ( \widehat{ v } ^i - v ^i )  \,\mathrm{d}^{m-1} x_\mu = \int _{ \partial K } \lambda \delta _{ i j } ( \widehat{ u } ^i - u ^i ) ( \widehat{ v } ^j - v ^j ) \mathbf{n} ^\mu  \,\mathrm{d}^{m-1} x_\mu .
  \end{equation*}
  Since $ v $ and $ \widehat{ v } $ are arbitrary, this can be written
  as
  \begin{equation*}
    \int _{ \partial K } \bigl[ ( \widehat{ \sigma } _i ^\mu - \sigma _i ^\mu ) \,\mathrm{d} ( \widehat{ u } ^i - u ^i ) \bigr]  \,\mathrm{d}^{m-1} x_\mu = \int _{ \partial K } \lambda \delta _{ i j } \bigl[ ( \widehat{ u } ^i - u ^i ) \,\mathrm{d} ( \widehat{ u } ^j - u ^j ) \bigr] \mathbf{n} ^\mu  \,\mathrm{d}^{m-1} x_\mu ,
  \end{equation*}
  and taking the exterior derivative of both sides yields
  \begin{multline*}
    \int _{ \partial K } \bigl[ \mathrm{d} ( \widehat{ \sigma } _i ^\mu - \sigma _i ^\mu ) \wedge \mathrm{d} ( \widehat{ u } ^i - u ^i ) \bigr]  \,\mathrm{d}^{m-1} x_\mu \\
    = \int _{ \partial K } \lambda \delta _{ i j } \bigl[ \mathrm{d} (
    \widehat{ u } ^i - u ^i ) \wedge \mathrm{d} ( \widehat{ u } ^j - u
    ^j ) \bigr] \mathbf{n} ^\mu \,\mathrm{d}^{m-1} x_\mu .
  \end{multline*}
  However, this vanishes by the symmetry of $ \delta $ and the
  antisymmetry of $\wedge $, so the multisymplecticity condition
  \eqref{eqn:msclJump} holds, as claimed.
\end{proof}

\begin{remark}
  More generally, we can replace
  $ \lambda \delta _{ i j } \in L ^\infty ( \partial K) $ above with
  penalty functions $ \lambda _{ i j } \in L ^\infty ( \partial K ) $
  such that $ \lambda _{ i j } = \lambda _{ j i } $ for
  $ i , j = 1 , \ldots, n $, and the argument above still holds. This
  same generalization applies to the penalty-based HDG methods we will
  encounter in \autoref{sec:hdg}.
\end{remark}

\begin{theorem}
  \label{thm:ncFlux}
  Suppose that, for all $ \tau \in \Sigma (K) $, there exists
  $ \widehat{ \tau } \in \widehat{ \Sigma } (\partial K) $ such that
  $ \widehat{ \tau } _i ^\mu \mathbf{n} _\mu = \tau _i ^\mu \mathbf{n}
  _\mu \rvert _{ \partial K } $ for $ i = 1 , \ldots, n $. Then the
  local flux function
  \begin{equation*}
    \Phi _K ( u , \sigma , \widehat{ u } , \widehat{ \sigma } ) = (\widehat{ u } - u ) \mathbf{n} 
  \end{equation*}
  is multisymplectic.
\end{theorem}

\begin{proof}
  The flux condition \eqref{eqn:numericalFlux} says that
  \begin{equation*}
    \int _{ \partial K } ( \widehat{ u } ^i - u ^i ) \widehat{ \tau } _i ^\mu \,\mathrm{d}^{m-1} x_\mu = 0,
  \end{equation*}
  for all $ \widehat{ \tau } \in \widehat{ \Sigma } ( \partial K )
  $. From the assumption on normal traces of elements of
  $ \Sigma (K) $, it follows that
  \begin{equation*}
    \int _{ \partial K } ( \widehat{ u } ^i - u ^i ) (\widehat{ \tau } _i ^\mu - \tau _i ^\mu )  \,\mathrm{d}^{m-1} x_\mu = 0,
  \end{equation*}
  for any $ \tau \in \Sigma (K) $ and
  $ \widehat{ \tau } \in \widehat{ \Sigma } ( \partial K ) $. This can
  be written as
  \begin{equation*}
    \int _{ \partial K } \bigl[ ( \widehat{ u } ^i - u ^i ) \,\mathrm{d}(\widehat{ \sigma  } _i ^\mu - \sigma  _i ^\mu ) \bigr]  \,\mathrm{d}^{m-1} x_\mu = 0,
  \end{equation*}
  and taking the exterior derivative yields \eqref{eqn:msclJump}.
\end{proof}

\subsection{Global multisymplecticity criteria}

Whenever \eqref{eqn:fluxHDG} is multisymplectic, we may of course sum
\eqref{eqn:msclHDG} over an arbitrary collection of elements
$ \mathcal{K} \subset \mathcal{T} _h $ to obtain the global statement
\begin{equation}
  \label{eqn:weakMSCL}
  \sum _{ K \in \mathcal{K} } \int _{ \partial K } (\mathrm{d} \widehat{ u } ^i \wedge \mathrm{d} \widehat{ \sigma } _i ^\mu) \,\mathrm{d}^{m-1} x_\mu = 0 ,
\end{equation}
or equivalently, by \autoref{lem:jump}
\begin{equation}
  \label{eqn:weakJumpMSCL}
  \sum _{ K \in \mathcal{K} } \int _{ \partial K } \bigl[ \mathrm{d} (\widehat{ u } ^i - u ^i ) \wedge \mathrm{d} ( \widehat{ \sigma } _i ^\mu - \sigma _i ^\mu ) \bigr] \,\mathrm{d}^{m-1} x_\mu = 0 .
\end{equation}
In the special case $ \mathcal{K} = \mathcal{T} _h $, for a
second-order linear elliptic PDE, \eqref{eqn:weakJumpMSCL} is
precisely Equation 2.10 from \citet{CoGoLa2009}, which establishes the
symmetry of the bilinear form used to solve for $ \widehat{ u } $ once
internal degrees of freedom have been eliminated (i.e., the Schur
complement). Hence, \eqref{eqn:weakJumpMSCL} can be seen as a
generalization of this symmetry condition to nonlinear multisymplectic
systems and arbitrary $ \mathcal{K} \subset \mathcal{T} _h $.

However, \eqref{eqn:weakMSCL} is a rather weak global
multisymplecticity condition, since it follows trivially from the
local condition \eqref{eqn:msclHDG}. We now define a stronger version
of multisymplecticity, which is more analogous to the classical
multisymplectic conservation law \eqref{eqn:msclIntegral}.

\begin{definition}
  The flux formulation \eqref{eqn:fluxHDG} is \emph{strongly
    multisymplectic} if solutions satisfy
  \begin{equation}
    \label{eqn:strongMSCL}
    \int _{ \partial ( \overline{ \bigcup \mathcal{K} } ) } (\mathrm{d} \widehat{ u } ^i \wedge \mathrm{d} \widehat{ \sigma } _i ^\mu) \,\mathrm{d}^{m-1} x_\mu = 0 ,
  \end{equation} 
  for all $ \mathcal{K} \subset \mathcal{T} _h $, whenever the
  $ C ^1 $ differential $1$-form
  $ \phi ^i _\mu \,\mathrm{d}\sigma _i ^\mu + f _i \,\mathrm{d}u ^i $
  is closed.
\end{definition}

Taking $ \mathcal{K} = \{ K \} $ immediately implies the local
condition \eqref{eqn:msclHDG} for each $ K \in \mathcal{T} _h $, so
the stronger condition \eqref{eqn:strongMSCL} indeed implies the
weaker condition \eqref{eqn:weakMSCL}. It follows that
\eqref{eqn:strongMSCL} holds if and only if the terms of
\eqref{eqn:weakMSCL} cancel on internal facets. The property that
conservation laws ``add up'' correctly over unions of elements is
directly related to the conservativity condition
\eqref{eqn:conservativity}; indeed, this is the reason the term
``conservative'' is used to describe numerical fluxes with
single-valued normal components. (See Equation 3.2 in
\citet{ArBrCoMa2001}.)

Let $ e = \partial K ^+ \cap \partial K ^- $ be an internal facet,
where $ K ^\pm \in \mathcal{T} _h $ are distinct. Recall from
\autoref{rmk:doubleValued} that an element
$ \widehat{ \tau } \in \widehat{ \Sigma } $ is generally double-valued
on $e$, since the $ \widehat{ \Sigma } ( \partial K ^+ ) $ and
$ \widehat{ \Sigma } ( \partial K ^-) $ components need not agree. As
is common in the discontinuous Galerkin literature (including
\citet{ArBrCoMa2001,CoGoLa2009}), we define the ``normal jump''
$ \llbracket \widehat{ \tau } \rrbracket \rvert _e \in \bigl[ L ^2 (e)
\bigr] ^n $ by
\begin{equation*}
  \llbracket \widehat{ \tau } \rrbracket _i \rvert _e = \widehat{ \tau } _i ^\mu \mathbf{n} _\mu \rvert _{ e ^+} + \widehat{ \tau } _i ^\mu \mathbf{n} _\mu \rvert _{ e ^-} .
\end{equation*} 
Here, $ e ^{ \pm } $ denotes that $e$ is oriented according to
$ \partial K ^\pm $, and we take the corresponding component of
$ \widehat{ \tau } $ and outer normal $ \mathbf{n} $ for each term on
the right-hand side. Denoting the set of internal facets of
$ \mathcal{T} _h $ by
$ \mathcal{E} _h ^\circ \coloneqq \mathcal{E} _h \setminus \partial U
$, we may sum directly over $ e \in \mathcal{E} _h ^\circ $ to define
$ \llbracket \widehat{ \tau } \rrbracket \in \bigl[ L ^2 ( \mathcal{E}
_h ^\circ ) \bigr] ^n $.  Hence, the normal component of
$ \widehat{ \tau } $ is single-valued on internal facets if and only
if $ \llbracket \widehat{ \tau } \rrbracket = 0 $.

With this notation, the conservativity condition
\eqref{eqn:conservativity} can be rewritten as
\begin{equation*}
  \int _{ \mathcal{E} _h ^\circ } \llbracket \widehat{ \sigma } \rrbracket _i \widehat{ v } ^i \,\mathrm{d}^{m-1} x = 0 , \qquad \forall \widehat{ v } \in \widehat{ V } _0 .
\end{equation*}
If the extension by zero of
$ \llbracket \widehat{ \sigma } \rrbracket $ to $ \mathcal{E} _h $ is
in $ \widehat{ V } _0 $, then applying the conservativity condition
with $ \widehat{ v } = \llbracket \widehat{ \sigma } \rrbracket $
immediately implies $ \llbracket \widehat{ \sigma } \rrbracket = 0 $,
and we say that $ \widehat{ \sigma } $ is \emph{strongly
  conservative}. However, this is not always the case: for example, if
$ \widehat{ \Sigma } ( \partial K ) $ contains discontinuous traces
but $ \widehat{ V } ( \partial K ) $ contains only continuous traces,
then in general
$ \llbracket \widehat{ \sigma } \rrbracket \notin \widehat{ V } _0 $,
so we cannot conclude that
$ \llbracket \widehat{ \sigma } \rrbracket $ vanishes. In this case,
we say that $ \widehat{ \sigma } $ is only \emph{weakly
  conservative}. (This terminology is taken from \citet{CoGoLa2009}.)

\begin{theorem}
  \label{thm:strongMSCL}
  If $ \widehat{ \sigma } \in \widehat{ \Sigma } $ satisfies the
  strong conservativity condition
  $ \llbracket \widehat{ \sigma } \rrbracket = 0 $, then for any
  $ \widehat{ u } \in \widehat{ V } $ and
  $ \mathcal{K} \subset \mathcal{T} _h $,
  \begin{equation*}
    \sum _{ K \in \mathcal{K} } \int _K ( \mathrm{d} \widehat{ u } ^i \wedge \mathrm{d} \widehat{ \sigma } _i ^\mu ) \,\mathrm{d}^{m-1} x_\mu = \int _{ \partial ( \overline{ \bigcup \mathcal{K} } ) } ( \mathrm{d} \widehat{ u } ^i \wedge \mathrm{d} \widehat{ \sigma } _i ^\mu ) \,\mathrm{d}^{m-1} x_\mu .
  \end{equation*}
  Consequently, if \eqref{eqn:fluxHDG} is multisymplectic and strongly
  conservative, then it is strongly multisymplectic.
\end{theorem}

\begin{proof}
  Let $ e = \partial K ^+ \cap \partial K ^- $ be an internal facet,
  where $ K ^\pm \subset \mathcal{K} $. Since
  $ \llbracket \widehat{ \sigma } \rrbracket = 0 $, we have
  \begin{equation*}
    \int _{ e ^+ } ( \widehat{ \sigma } _i ^\mu \,\mathrm{d} \widehat{ u } ^i ) \,\mathrm{d}^{m-1} x_\mu + \int _{ e ^- } ( \widehat{ \sigma } _i ^\mu \,\mathrm{d} \widehat{ u } ^i ) \,\mathrm{d}^{m-1} x_\mu = 
    \int _e \bigl( \llbracket \widehat{ \sigma } \rrbracket _i \,\mathrm{d}\widehat{ u } ^i \bigr) \,\mathrm{d}^{m-1} x = 0 . 
  \end{equation*}
  Finally, taking the exterior derivative implies that
  \begin{equation*}
    \int _{ e ^+ } ( \mathrm{d} \widehat{ u } ^i \wedge \mathrm{d} \widehat{ \sigma } _i ^\mu ) \,\mathrm{d}^{m-1} x_\mu + \int _{ e ^- } ( \mathrm{d} \widehat{ u } ^i \wedge \mathrm{d} \widehat{ \sigma } _i ^\mu ) \,\mathrm{d}^{m-1} x_\mu = 0,
  \end{equation*}
  so the contributions from internal facets vanish, as claimed.
\end{proof}

\subsection{The flux formulation for exact solutions}
\label{sec:exactSolutions}

We now apply the theory of the preceding sections to exact solutions
of \eqref{eqn:pdeSystem}, in the sense of distributions, in which case
the flux formulation \eqref{eqn:fluxHDG} consists of
infinite-dimensional function spaces.

\begin{definition}
  The \emph{exact flux formulation} on $ \mathcal{T} _h $ is the flux
  formulation \eqref{eqn:fluxHDG} associated to the function spaces
  \begin{align*}
    V (K) &= \bigl[ H ^2 (K) \bigr] ^n ,  & \Sigma (K) &= \bigl[ H ^1 (K) \bigr] ^{ m n } ,\\
    \widehat{ V } &= \bigl[ L ^2 ( \mathcal{E} _h  ) \bigr] ^n , & \widehat{ \Sigma } ( \partial K ) &= \bigl[ L ^2 ( \partial K ) \bigr] ^{ m n } ,
  \end{align*}
  along with the flux functions
  $ \Phi _K ( u, \sigma , \widehat{ u } , \widehat{ \sigma } ) =
  \widehat{ \sigma } - \sigma $, for $ K \in \mathcal{T} _h $.
\end{definition}

The next theorem uses a domain-decomposition-type argument to relate
the exact flux formulation to solutions of \eqref{eqn:pdeSystem}, in
the sense of distributions, defined over certain global function
spaces on $U$.

\begin{theorem}
  \label{thm:exactSolution}
  The element
  $ ( u, \sigma , \widehat{ u } , \widehat{ \sigma } ) \in V \times
  \Sigma \times \widehat{ V } \times \widehat{ \Sigma } $ is a
  solution to the exact flux formulation if and only if
  \begin{equation*}
    ( u , \sigma ) \in \Bigl( V \cap \bigl[ H ^1 (U) \bigr] ^n \Bigr)
    \times \Bigl( \Sigma \cap \bigl[ H ( \operatorname{div}; U ) \bigr]
    ^n \Bigr)
  \end{equation*}
  is a solution to \eqref{eqn:pdeSystem} on $U$, in the sense of
  distributions, and $ ( \widehat{ u } , \widehat{ \sigma } ) $ are
  the exact traces $ \widehat{ u } = u \rvert _{ \mathcal{E} _h } $
  and
  $ \widehat{ \sigma } \rvert _{ \partial K } = \sigma \rvert
  _{ \partial K } $ for all $ K \in \mathcal{T} _h $.
\end{theorem}

\begin{proof}
  Suppose
  $ ( u, \sigma , \widehat{ u } , \widehat{ \sigma } ) \in V \times
  \Sigma \times \widehat{ V } \times \widehat{ \Sigma } $ is a
  solution to the exact flux formulation.

  In particular, \eqref{eqn:localSolverA} holds for all
  $ \tau \in \bigl[ C _c ^\infty (K) \bigr] ^{ m n } $, which
  immediately gives $ \partial _\mu u ^i = \phi ^i _\mu $ on $K$, in
  the sense of distributions. Taking more general test functions
  $ \tau \in \bigl[ C ^\infty (K) \bigr] ^{ m n } $, not necessarily
  vanishing on $ \partial K $, implies that
  $ \widehat{ u } \rvert _{ \partial K } = u \rvert _{ \partial K } $
  in the trace sense. Since this holds for all
  $ K \in \mathcal{T} _h $, it follows that
  $ \widehat{ u } = u \rvert _{ \mathcal{E} _h } $. Hence, the trace
  of $u$ is single-valued on $ \mathcal{E} _h $, and we may therefore
  conclude (cf.~\citet[Proposition III.1.1]{BrFo1991}) that
  $ u \in \bigl[ H ^1 (U) \bigr] ^n $.

  Similarly, taking smooth test functions $v$ in
  \eqref{eqn:localSolverB} implies that
  $ - \partial _\mu \sigma _i ^\mu = f _i $ holds in the sense of
  distributions, and
  $ \widehat{ \sigma } _i ^\mu \mathbf{n} _\mu \rvert _{ \partial K }
  = \sigma _i ^\mu \mathbf{n} _\mu \rvert _{ \partial K } $ holds in
  the trace sense, for all $ K \in \mathcal{T} _h $. The flux equation
  \eqref{eqn:numericalFlux} implies further that
  $ \widehat{ \sigma } \rvert _{ \partial K } = \sigma \rvert
  _{ \partial K } $, and the conservativity condition
  \eqref{eqn:conservativity} gives
  $ \llbracket \widehat{ \sigma } \rrbracket = 0 $ on
  $ \mathcal{E} _h ^\circ $. Hence, the normal trace of $\sigma$ is
  single-valued on $ \mathcal{E} _h $, and we may therefore conclude
  (cf.~\citet[Proposition III.1.2]{BrFo1991}) that
  $ \sigma \in \bigl[ H (\operatorname{div}; U) \bigr] ^n $.

  The converse is a simple verification of \eqref{eqn:fluxHDG}. If
  $ u \in V \cap \bigl[ H ^1 (U) \bigr] ^n $, then it has a
  (single-valued) trace
  $ \widehat{ u } \in \bigl[ L ^2 ( \mathcal{E} _h ) \bigr] ^n =
  \widehat{ V } $. Likewise, if
  $ \sigma \in \Sigma \cap \bigl[ H (\operatorname{div}; U) \bigr] ^n
  $, then it has a trace
  $ \widehat{ \sigma } \rvert _{ \partial K } \in \bigl[ L ^2
  ( \partial K ) \bigr] ^{ m n } = \widehat{ \Sigma } (\partial K) $ for each
  $ K \in \mathcal{T} _h $; this satisfies
  $ \llbracket \widehat{ \sigma } \rrbracket = 0 $ on
  $ \mathcal{E} _h ^\circ $, so
  \eqref{eqn:numericalFlux}--\eqref{eqn:conservativity} hold. Finally,
  equations \eqref{eqn:localSolverA}--\eqref{eqn:localSolverB} hold by
  the assumption that $ ( u, \sigma ) $ satisfies
  \eqref{eqn:pdeSystem} in the sense of distributions.
\end{proof}

\begin{corollary}
  \label{cor:exactMSCL}
  The exact flux formulation satisfies
  \begin{equation*}
    \int _{ \partial ( \overline{ \bigcup \mathcal{K} } ) } (\mathrm{d} u ^i \wedge \mathrm{d} \sigma _i ^\mu) \,\mathrm{d}^{m-1} x_\mu = 0 ,
  \end{equation*}
  for all $ \mathcal{K} \subset \mathcal{T} _h $, whenever
  $ \phi ^i _\mu \,\mathrm{d}\sigma _i ^\mu + f _i \,\mathrm{d}u ^i $
  is closed.
\end{corollary}

\begin{proof}
  It follows immediately from \autoref{thm:ldgFlux} (with
  $ \lambda \equiv 0 $) and \autoref{thm:strongMSCL} that the exact
  flux formulation is strongly multisymplectic, so it satisfies
  \eqref{eqn:strongMSCL}. Moreover, \autoref{thm:exactSolution} gives
  $ \widehat{ u } = u \rvert _{ \mathcal{E} _h } $ and
  $ \widehat{ \sigma } \rvert _{ \partial K } = \sigma \rvert
  _{ \partial K } $ for all $ K \in \mathcal{T} _h $, so we may
  ``remove the hats'' from \eqref{eqn:strongMSCL}.
\end{proof}

\section{Multisymplecticity of particular HDG methods}
\label{sec:hdg}

We now apply the results of \autoref{sec:flux} to the particular HDG
methods discussed in \citet{CoGoLa2009}. Although the flux formulation
\eqref{eqn:fluxHDG} is more general than that for the class of
second-order linear elliptic PDEs they consider, the spaces and fluxes
used to define the methods are essentially unchanged.

Throughout this section, we assume that $U \subset \mathbb{R}^m $ is
polyhedral and that $ \mathcal{T} _h $ is a triangulation of $U$ by
$ m $-simplices. We denote by $ \mathcal{P} _r (K) $ the space of
degree-$r$ polynomials on $ K \in \mathcal{T} _h $ and by
$ \mathcal{P} _r ( e ) $ the space of degree-$r$ polynomials on
$ e \in \mathcal{E} _h $. We also define spaces of discontinuous
polynomial boundary traces,
\begin{equation*}
  \mathcal{P} _r ( \mathcal{E} _h ) \coloneqq \bigl\{ \widehat{ w } \in L ^2 ( \mathcal{E} _h ) : \widehat{ w } \rvert _e \in \mathcal{P} _r (e) ,\ \forall e \in \mathcal{E} _h ^\circ \bigr\} 
\end{equation*}
and continuous polynomial boundary traces,
\begin{equation*}
  \mathcal{P} _r ^c ( \mathcal{E} _h ) \coloneqq \bigl\{ \widehat{ w } \in C ^0  ( \mathcal{E} _h ) : \widehat{ w } \rvert _e \in \mathcal{P} _r (e) ,\ \forall e \in \mathcal{E} _h \bigr\} ,
\end{equation*}
which in \citet{CoGoLa2009} are called $ \mathcal{M} _{ h , r } $ and
$ \mathcal{M} _{ h , r } ^c $, respectively.

\subsection{The RT-H method} The hybridized Raviart--Thomas (RT-H)
method uses the local function spaces
\begin{equation*}
  V (K) = \bigl[ \mathcal{P} _r (K) \bigr] ^n , \qquad \Sigma (K) = \bigl[ \mathcal{P} _r (K) ^m + x \mathcal{P} _r (K) \bigr] ^n ,
\end{equation*}
i.e., degree-$r$ Lagrange finite elements and Raviart--Thomas finite
elements, respectively, for each $ i = 1, \ldots, n $. The trace
spaces are taken to be
\begin{equation*}
  \widehat{ V } = \bigl[ \mathcal{P} _r ( \mathcal{E} _h ) \bigr] ^n , \qquad \widehat{ \Sigma } (\partial K) = \bigl[ L ^2 ( \partial K ) \bigr] ^{ m n } ,
\end{equation*}
and the local flux functions are
\begin{equation*}
  \Phi _K ( u , \sigma , \widehat{ u } , \widehat{ \sigma } ) = \widehat{ \sigma } - \sigma .
\end{equation*}
Note that, although $ \widehat{ \Sigma } ( \partial K ) $ is
infinite-dimensional, the flux condition \eqref{eqn:numericalFlux}
simply states that
$ \widehat{ \sigma } \rvert _{ \partial K } = \sigma \rvert
_{ \partial K } $, so we may eliminate this equation and substitute
$\sigma$ wherever $ \widehat{ \sigma } $ appears in the remaining
equations.

\begin{theorem}
  \label{thm:rt-h}
  The RT-H method is strongly multisymplectic.
\end{theorem}

\begin{proof}
  Observe that $ \Phi _K $ is a special case of the flux in
  \autoref{thm:ldgFlux}, with $ \lambda \equiv 0 $. Since
  $ \widehat{ \Sigma } ( \partial K ) = \bigl[ L ^2 ( \partial K )
  \bigr] ^{m n} $, it follows that for any $ v \in V (K) $ and
  $ \widehat{ v } \in \widehat{ V } $, we have
  $ ( \widehat{ v } - v ) \mathbf{n} \rvert _{ \partial K } \in
  \widehat{ \Sigma } ( \partial K ) $. Hence, the hypotheses of
  \autoref{thm:ldgFlux} are satisfied, so the method is
  multisymplectic.

  To show strong multisymplecticity, let
  $ e \in \mathcal{E} _h ^\circ $ be an arbitrary internal facet, and
  write $ e = \partial K ^+ \cap \partial K ^- $. Since
  $ \sigma \rvert _{K^\pm} \in \Sigma ( K ^\pm ) $, a standard result
  on Raviart--Thomas elements (cf.~\citet[Proposition
  III.3.2]{BrFo1991}) implies that
  $ \sigma _i ^\mu \mathbf{n} _\mu \rvert _{ e^\pm } \in \mathcal{P}
  _r (e) $ for $ i = 1 , \ldots, n $.\footnote{This result holds since
    $e$ lies in an affine hyperplane in $ \mathbb{R}^m $, so
    $ x \cdot \mathbf{n} $ is constant on $e$. It follows that the
    degree-$( r + 1 )$ elements of the Raviart--Thomas space
    nevertheless have degree-$r$ normal traces.} The flux condition
  \eqref{eqn:numericalFlux} implies
  $ \widehat{ \sigma } \rvert _{ e ^\pm } = \sigma \rvert _{ e ^\pm }
  $, so it follows from the above that
  $ \llbracket \widehat{ \sigma } \rrbracket \rvert _e \in \bigl[
  \mathcal{P} _r (e) \bigr] ^n $. Since this holds for all
  $ e \in \mathcal{E} _h ^\circ $, the extension by zero of
  $ \llbracket \widehat{ \sigma } \rrbracket $ to $ \mathcal{E} _h $
  is in $ \widehat{ V } _0 $.  Therefore, the RT-H method is strongly
  conservative, so \autoref{thm:strongMSCL} implies that it is
  strongly multisymplectic.
\end{proof}

\begin{remark}
  To prove multisymplecticity, instead of using \autoref{thm:ldgFlux},
  we could simply have used
  $ ( \widehat{ \sigma } - \sigma ) \rvert _{ \partial K } = 0 $ to
  see that \eqref{eqn:msclJump} holds. However, the argument we have
  developed here is more general, and we will see in \autoref{sec:ldg}
  that it also applies to methods where $ \lambda \not\equiv 0 $.
\end{remark}

\subsection{The BDM-H method}

The hybridized Brezzi--Douglas--Marini (BDM-H) method uses the local
function spaces
\begin{equation*}
  V (K) = \bigl[ \mathcal{P} _{ r -1 } (K) \bigr] ^n , \qquad \Sigma (K) = \bigl[ \mathcal{P} _r (K) \bigr] ^{ m n } ,
\end{equation*}
i.e., degree-$(r-1)$ Lagrange finite elements and degree-$r$
Brezzi--Douglas--Marini finite elements, respectively, for each
$ i = 1 , \ldots, n $. As in the RT-H method, the trace spaces are
taken to be
\begin{equation*}
  \widehat{ V } = \bigl[ \mathcal{P} _r ( \mathcal{E} _h ) \bigr] ^n , \qquad \widehat{ \Sigma } ( \partial K ) = \bigl[ L ^2 ( \partial K ) \bigr] ^{ m n } ,
\end{equation*}
and the local flux functions are
\begin{equation*}
  \Phi _K ( u , \sigma, \widehat{ u } , \widehat{ \sigma } ) = \widehat{ \sigma } - \sigma .
\end{equation*}
As with RT-H, the flux condition \eqref{eqn:numericalFlux} states that
$ \widehat{ \sigma } \rvert _{ \partial K } = \sigma \rvert
_{ \partial K } $, so we may eliminate \eqref{eqn:numericalFlux} and
substitute $\sigma$ for $ \widehat{ \sigma } $ in the remaining
equations.

\begin{theorem}
  \label{thm:bdm-h}
  The BDM-H method is strongly multisymplectic.
\end{theorem}

\begin{proof}
  Since the trace spaces and flux function are the same as the RT-H
  method, multisymplecticity follows exactly as in \autoref{thm:rt-h}.

  Given any internal facet $ e = \partial K ^+ \cap \partial K ^- $, we
  have
  $ \sigma \rvert _{ K ^\pm } \in \bigl[ \mathcal{P} _r ( K ^\pm )
  \bigr] ^{ m n } $, so
  $ \widehat{ \sigma } \rvert _{ e ^\pm } = \sigma \rvert _{ e ^\pm }
  \in \bigl[ \mathcal{P} _r ( e ) \bigr] ^{ m n } $ and therefore
  $ \llbracket \widehat{ \sigma } \rrbracket \rvert _e \in \bigl[
  \mathcal{P} _r (e) \bigr] ^n $. Since this holds for all
  $ e \in \mathcal{E} _h ^\circ $, the extension by zero of
  $ \llbracket \widehat{ \sigma } \rrbracket $ to $ \mathcal{E} _h $
  is in $ \widehat{ V } _0 $.  Therefore, the BDM-H method is strongly
  conservative, so \autoref{thm:strongMSCL} implies that it is
  strongly multisymplectic.
\end{proof}

\subsection{The LDG-H methods}
\label{sec:ldg}

There are three variants of the hybridized local discontinuous
Galerkin method (LDG-H) discussed in \citet{CoGoLa2009}, corresponding
to different choices of the local function spaces. In the setting
and notation considered here, these three pairs of spaces are:
\begin{subequations}
\begin{alignat}{2}
  V (K) &= \bigl[ \mathcal{P} _{ r -1 } (K) \bigr] ^n , \qquad & \Sigma (K) &= \bigl[ \mathcal{P} _r (K) \bigr] ^{m n} , \label{eqn:ldg1} \\
  V (K) &= \bigl[ \mathcal{P} _{ r } (K) \bigr] ^n , \qquad& \Sigma (K) &= \bigl[ \mathcal{P} _r (K) \bigr] ^{m n} , \label{eqn:ldg2}\\
  V (K) &= \bigl[ \mathcal{P} _{ r } (K) \bigr] ^n , \qquad& \Sigma (K)
  &= \bigl[ \mathcal{P} _{r-1} (K) \bigr] ^{m n} \label{eqn:ldg3} .
\end{alignat} 
\end{subequations}
Whichever of these we choose, the trace spaces are
\begin{equation*}
  \widehat{ V } = \bigl[ \mathcal{P} _r ( \mathcal{E} _h ) \bigr] ^n , \qquad \widehat{ \Sigma } ( \partial K ) = \bigl[ L ^2 ( \partial K ) \bigr] ^{ m n } ,
\end{equation*}
and the local flux functions are
  \begin{equation*}
    \Phi _K ( u , \sigma , \widehat{ u } , \widehat{ \sigma } ) = ( \widehat{ \sigma } - \sigma ) - \lambda ( \widehat{ u } - u ) \mathbf{n} ,
  \end{equation*}
  where the penalty function $\lambda$ is piecewise constant on
  $ \partial K $, i.e., constant on each facet. Note that since
  $\lambda$ may be different for each $K \in \mathcal{T} _h $, on an
  internal facet $ e = \partial K ^+ \cap \partial K ^- $, the
  constants $ \lambda \rvert _{ e ^\pm } $ need not be equal to one
  another.  The flux condition \eqref{eqn:numericalFlux} states that
  \begin{equation*}
    \widehat{ \sigma } \rvert _{ \partial K } = \bigl[ \sigma + \lambda ( \widehat{ u } - u ) \mathbf{n} \bigr] \rvert _{ \partial K } ,
  \end{equation*}
  so we may eliminate \eqref{eqn:numericalFlux} and substitute the
  expression on the right-hand side wherever $ \widehat{ \sigma } $
  appears in the remaining equations.

\begin{theorem}
  The LDG-H method is strongly multisymplectic for each of the choices
  \eqref{eqn:ldg1}--\eqref{eqn:ldg3} of local function spaces.
\end{theorem}

\begin{proof}
  The flux $ \Phi _K $ is precisely that of \autoref{thm:ldgFlux}. As
  shown in the proof of \autoref{thm:rt-h}, the space
  $ \widehat{ \Sigma } ( \partial K ) = \bigl[ L ^2 ( \partial K )
  \bigr] ^{ m n } $ satisfies the hypotheses of \autoref{thm:ldgFlux},
  so the LDG-H method is multisymplectic.

  Now, for each of \eqref{eqn:ldg1}--\eqref{eqn:ldg3}, we have
  \begin{equation*}
    V (K) \subset \bigl[ \mathcal{P} _r (K) \bigr] ^n , \qquad \Sigma (K) \subset \bigl[ \mathcal{P} _r (K) \bigr] ^{ m n } .
  \end{equation*}
  For any internal facet $ e = \partial K ^+ \cap \partial K ^- $,
  since $ \lambda \rvert _{ e ^\pm } $ are constants, it follows that
  \begin{equation*}
    \widehat{ \sigma } \rvert _{ e ^\pm } = \bigl[ \sigma + \lambda ( \widehat{ u } - u ) \mathbf{n} \bigr] \rvert _{ e ^\pm } \in \bigl[ \mathcal{P} _r (e) \bigr] ^{ m n } ,
  \end{equation*}
  so
  $ \llbracket \widehat{ \sigma } \rrbracket \rvert _e \in \bigl[
  \mathcal{P} _r (e) \bigr] ^n $. (Note that if
  $ \lambda \rvert _{ e ^\pm } $ were arbitrary $ L ^\infty $ penalty
  functions, rather than constants, this would not necessarily be
  true; in fact, $ \widehat{ \sigma } \rvert _{ e ^\pm } $ might not
  be polynomials at all.)  Since this holds for all
  $ e \in \mathcal{E} _h ^\circ $, the extension by zero of
  $ \llbracket \widehat{ \sigma } \rrbracket $ to $ \mathcal{E} _h $
  is in $ \widehat{ V } _0 $. Therefore, the LDG-H method is strongly
  conservative, so \autoref{thm:strongMSCL} implies that it is
  strongly multisymplectic.
\end{proof}

\subsection{The CG-H method}

The hybridized continuous Galerkin (CG-H) method uses the local function spaces
\begin{equation*}
  V (K) = \bigl[ \mathcal{P} _r (K) \bigr] ^n , \qquad \Sigma (K) = \bigl[ \mathcal{P} _{ r -1 } (K) \bigr] ^{ m n } ,
\end{equation*}
i.e., degree-$r$ Lagrange finite elements for $ u ^i , v ^i $ and
degree-$ (r -1 ) $ Lagrange finite elements for
$ \sigma _i ^\mu , \tau _i ^\mu $, for $ \mu = 1 , \ldots, m $ and
$ i = 1, \ldots, n $. The trace spaces are taken to be
\begin{equation*}
  \widehat{ V } = \bigl[ \mathcal{P} _r ^c (\mathcal{E} _h ) \bigr] ^n , \qquad \widehat{ \Sigma } (\partial K) = \bigl\{ v \mathbf{n} \rvert _{ \partial K } : v \in V (K) \bigr\} ,
\end{equation*}
and the local flux functions are
\begin{equation*}
  \Phi _K ( u, \sigma, \widehat{ u } , \widehat{ \sigma } ) = ( \widehat{ u } - u ) \mathbf{n} .
\end{equation*}
Since $ u \in V (K) $, we immediately have
$ u \mathbf{n} \rvert _{ \partial K } \in \widehat{ \Sigma } (\partial
K) $. Moreover, since $ \widehat{ V } $ consists of \emph{continuous}
polynomials, the degrees of freedom for
$ \widehat{ V } (\partial K ) $ are a subset of those for $ V (K) $,
so
$ \widehat{ u } \mathbf{n} \rvert _{ \partial K } \in \widehat{ \Sigma
} (\partial K) $ as well. Hence, taking the flux condition
\eqref{eqn:numericalFlux} with
$ \widehat{ \tau } = ( \widehat{ u } - u ) \mathbf{n} \rvert
_{ \partial K } $ implies
$ \widehat{ u } \rvert _{ \partial K } = u \rvert _{ \partial K } $
for all $ K \in \mathcal{T} _h $.

\begin{remark}
  \label{rmk:semilinearCG-H}
  The CG-H method is so named because it coincides with the classical
  continuous Galerkin method with Lagrange finite elements when
  applied to second-order linear elliptic PDEs of the type considered
  in \autoref{ex:semilinearIntro}, \autoref{ex:semilinearMSCLdiff},
  and \autoref{ex:semilinearMSCLint}, as long as
  $ a = a ^{ \mu \nu } (x) $ is constant on each
  $ K \in \mathcal{T} _h $. See
  \citet{CoGoWa2007,CoGoLa2009,Cockburn2016}, where the relationship
  between CG-H and the technique of ``static condensation'' is also
  discussed.

  More generally, this correspondence also holds for the semilinear system
  \begin{equation*}
    \operatorname{grad} u = a ^{-1} \sigma , \qquad - \operatorname{div} \sigma  = f ( \cdot , u ) .
  \end{equation*}
  Indeed, since \eqref{eqn:numericalFlux} implies
  $ \widehat{ u } \rvert _{ \partial K } = u \rvert _{ \partial K } $,
  substituting $u$ for $ \widehat{ u } $ in \eqref{eqn:localSolverA}
  implies $ \sigma = a \operatorname{grad} u $, as long as $a$ is
  constant on each $ K \in \mathcal{T} _h $. (Otherwise,
  $ a \operatorname{grad} u $ is generally not in $\Sigma$.)  It
  follows that $ u \in C ^0 (U) \cap V $, so for all test functions
  $v$ in this same space vanishing on $ \partial U $,
  summing \eqref{eqn:localSolverB} over $ K \in \mathcal{T} _h $ gives
  \begin{align*}
    \int _U a \operatorname{grad} u \cdot \operatorname{grad} v
    &= \sum _{ K \in \mathcal{T} _h } \int _K a \operatorname{grad} u \cdot \operatorname{grad} v \\
    &= \sum _{ K \in \mathcal{T} _h } \biggl[ \int _K f (\cdot , u ) v + \int _{ \partial K } \widehat{ \sigma } v \cdot \mathbf{n} \biggr] \\
    &= \int _U f ( \cdot , u ) v + \sum _{ K \in \mathcal{T} _h } \int _{ \partial K } \widehat{ \sigma } v \cdot \mathbf{n} \\
    &= \int _U f (\cdot , u ) v .
  \end{align*}
  The boundary term vanishes by \eqref{eqn:conservativity}, since the
  assumption that $v$ is continuous and vanishes on $ \partial U $
  implies $ v \rvert _{ \mathcal{E} _h } \in \widehat{ V } _0
  $. Hence, $u$ is a solution to the continuous Galerkin method for
  $ - \operatorname{div} ( a \operatorname{grad} u ) = f ( \cdot , u )
  $.
\end{remark}

We now prove that the CG-H method is multisymplectic,
although---unlike the other HDG methods considered here---it is
\emph{not strongly multisymplectic} except in dimension $ m = 1 $,
when multisymplecticity is just ordinary symplecticity.

\begin{theorem}
  \label{thm:cg-h}
  The CG-H method is multisymplectic. It is strongly multisymplectic
  (i.e., symplectic) when $ m = 1 $.
\end{theorem}

\begin{proof}
  Since $ ( \widehat{ u } - u ) \rvert _{ \partial K } = 0 $, we see
  directly that \eqref{eqn:msclJump} holds, so multisymplecticity
  follows by \autoref{lem:jump}.

  When $ m = 1 $, facets are simply vertices, so $ \mathcal{E} _h $ is
  discrete and finite, and the continuity conditions on
  $ \widehat{ V } $ are trivial. Hence,
  $ \bigl[ \mathcal{P} _r ^c ( \mathcal{E} _h ) \bigr] ^n = \bigl[
  \mathcal{P} _r ( \mathcal{E} _h ) \bigr] ^n = \mathbb{R} ^{ \lvert
    \mathcal{E} _h \rvert n } $. The result follows immediately from
  the trivial observation that
  $ \llbracket \widehat{ \sigma } \rrbracket \rvert _e \in \mathbb{R}
  ^n $ at each internal vertex $e$.
\end{proof}

\begin{remark}
  Although $ \Phi _K $ is the same flux function as in
  \autoref{thm:ncFlux}, the hypotheses of that theorem do not hold for
  the CG-H method. Here, $ \widehat{ \Sigma } ( \partial K ) $
  consists only of $ \widehat{ \tau } $ whose normal traces are
  continuous on $ \partial K $, while this is not necessarily true of
  $ \tau \rvert _{ \partial K } $ for arbitrary
  $ \tau \in \Sigma (K) $.
\end{remark}

\begin{example}[CG-H for Laplace's equation in $ \mathbb{R}^2 $]
  \label{ex:laplaceCG-H}
  Let $ m = 2 $, $ n = 1 $, and consider the mixed form of Laplace's
  equation,
  \begin{equation*}
    \operatorname{grad} u = \sigma , \qquad - \operatorname{div} \sigma = 0 .
  \end{equation*}
  Let us apply the lowest-order CG-H method, with $ r = 1 $, so that
  for each $ K \in \mathcal{T} _h $, we have
  $ V (K) = \mathcal{P} _1 (K) $ and
  $ \Sigma (K) = \bigl[ \mathcal{P} _0 (K) \bigr] ^2 $.

  As discussed in \autoref{rmk:semilinearCG-H}, we have
  $ \widehat{ u } \rvert _{ \partial K } = u \rvert _{ \partial K } $
  and $ \sigma \rvert _{ K } = \operatorname{grad} u \rvert _K
  $. Therefore, $ \widehat{ \sigma } \rvert _{ \partial K } $ is
  determined by \eqref{eqn:localSolverB}, which in this case states
  \begin{equation*}
    \int _{ \partial K } \widehat{ \sigma } v \cdot \mathbf{n} = \int _K \operatorname{grad} u \cdot \operatorname{grad} v , \qquad \forall v \in V (K) .
  \end{equation*}
  Since
  $ \widehat{ \sigma } \rvert _{ \partial K } = w \mathbf{n} \rvert
  _{ \partial K } $ for some $ w \in V (K) $, we may rewrite this as
  \begin{equation}
    \label{eqn:sigmaHatLaplace}
    \int _{ \partial K } w v = \int _K \operatorname{grad} u \cdot \operatorname{grad} v , \qquad \forall v \in V (K) .
  \end{equation}
  However, $w \rvert _{ \partial K } $ is generally \emph{not} equal
  to
  $ \operatorname{grad} u \cdot \mathbf{n} \rvert _{ \partial K } $,
  since the latter is piecewise constant (and generally discontinuous)
  on $ \partial K $, whereas the former must be continuous and
  linear. Instead, we must set up a linear system and solve for $w$ in
  terms of $u$.

  For simplicity, let us suppose that $K$ is isometric to the
  standard, equilateral reference triangle in $\mathbb{R}^3$, defined
  by
  \begin{equation*}
    T \coloneqq \bigl\{  ( x, y , z ) \in \mathbb{R}^3 _{ \geq 0 } : x + y + z = 1 \bigr\} .
  \end{equation*}
  On $T$, the Lagrange basis of linear ``hat functions'' simply
  consists of the coordinate functions $x$, $y$, and $z$, and we can write
  \begin{equation*}
    u = u _1 x + u _2 y + u _3 z , \qquad w = w _1 x + w _2 y + w _3 z .
  \end{equation*}
  Solving the linear system corresponding to
  \eqref{eqn:sigmaHatLaplace} yields, after a calculation,
  \begin{align*}
    w _1 &= \frac{ \sqrt{ 6 }}{ 6 } ( 2 u _1 - u _2 - u _3 ) ,\\
    w _2 &= \frac{ \sqrt{ 6 }}{ 6 } (- u _1 +2 u _2 - u _3 ) ,\\
    w _3 &= \frac{ \sqrt{ 6 }}{ 6 } ( - u _1 - u _2 +2 u _3 ) .
  \end{align*}
  The multisymplectic form restricted to $ \partial T $ is therefore
  $ (\mathrm{d} u \wedge \mathrm{d} w ) \mathbf{n} \rvert _{ \partial
    T } $, so the multisymplectic conservation law states that
  $ \int _{ \partial T } \mathrm{d} u \wedge \mathrm{d} w = 0 $.

  Let $ e _{ i j } $ denote the edge in $ \partial T $ going from the
  $i$th standard basis vector to the $j$th standard basis vector in
  $\mathbb{R}^3 $. Using the above expressions for $w$ in terms of
  $u$, another calculation shows that
  \begin{subequations}
    \label{eqn:cg-hBoundaryIntegrals}
  \begin{align}
    \int _{ e _{ 1 2 } } \mathrm{d} u \wedge \mathrm{d} w &= \frac{ \sqrt{ 3 } }{ 6 } ( \mathrm{d} u _3 \wedge \mathrm{d} u _1 - \mathrm{d} u _2 \wedge \mathrm{d} u _3 ), \\
    \int _{ e _{ 2 3 } } \mathrm{d} u \wedge \mathrm{d} w &= \frac{ \sqrt{ 3 } }{ 6 } ( \mathrm{d} u _1 \wedge \mathrm{d} u _2 - \mathrm{d} u _3 \wedge \mathrm{d} u _1 ) ,    \\
    \int _{ e _{ 3 1 } } \mathrm{d} u \wedge \mathrm{d} w &= \frac{ \sqrt{ 3 } }{ 6 } ( \mathrm{d} u _2 \wedge \mathrm{d} u _3 - \mathrm{d} u _1 \wedge \mathrm{d} u _2 )  ,
  \end{align}
\end{subequations}
  from which it is immediately apparent that
  \begin{equation*}
    \int _{ \partial T } \mathrm{d} u \wedge \mathrm{d} w = \int _{ e _{ 1 2 } } \mathrm{d} u \wedge \mathrm{d} w + \int _{ e _{ 2 3 } } \mathrm{d} u \wedge \mathrm{d} w + \int _{ e _{ 3 1 } } \mathrm{d} u \wedge \mathrm{d} w = 0 ,
  \end{equation*}
  so the method is indeed multisymplectic on $T$, following
  \autoref{thm:cg-h}.
\end{example}

We now show that strong multisymplecticity does not hold for the CG-H
method when $ m > 1 $, so the result of \autoref{thm:cg-h} is the best
we can hope for.  The proof uses a counterexample based on
\autoref{ex:laplaceCG-H}.

\begin{proposition}
  The CG-H method is not strongly multisymplectic.
\end{proposition}

\begin{proof}
  \begin{figure}
    \centering
  \begin{tikzpicture}[scale=2]
    \node (1) [circle, fill, label=left:{$ u _1 $}] at (-{sqrt(3/2)}, 0) {};
    \node (2) [circle, fill, label=below:{$ u _2 $}] at (0, {-sqrt(2)/2}) {};
    \node (3) [circle, fill, label=above:{$ u _3 $}] at (0, {sqrt(2)/2}) {};
    \node (4) [circle, fill, label=right:{$ u _4 $}] at ({sqrt(3/2)}, 0) {};
    \draw [semithick] (1)--(2) node [pos=.4, label=below:{$ e _{ 1 2 } $}] {}
    --(4) node [pos=.6, label=below:{$ e _{ 2 4 } $}] {}
    --(3) node [pos=.4, label=above:{$ e _{ 3 4 } $}] {}
    --(1) node [pos=.6, label=above:{$ e _{ 1 3 } $}] {};
    \draw [semithick] (2)--(3) node [midway, label={[label distance=-.5em]right:{$ e _{ 2 3 } $}}] {};
  \end{tikzpicture}
  \caption{A triangulation on which the CG-H method is not strongly
    multisymplectic for Laplace's equation.\label{fig:twoTriangles}}
  \end{figure}
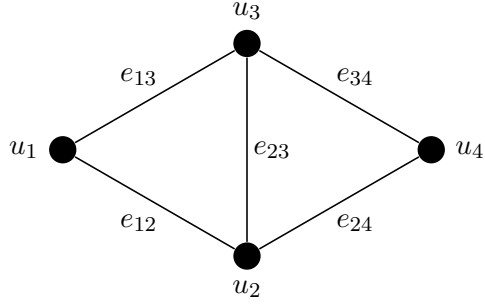

  Consider the mixed form of Laplace's equation on the domain
  $U \subset \mathbb{R}^2 $ triangulated by two equilateral triangles,
  as shown in \autoref{fig:twoTriangles}. For simplicity, as in
  \autoref{ex:laplaceCG-H}, we suppose that each of these is isometric
  to the standard reference triangle $T$, so that the edge lengths are
  all $ \sqrt{ 2 } $. Note that $ \widehat{ V } _0 = \{ 0 \} $, since
  the degrees of freedom for $ \widehat{ V } $ all lie on
  $ \partial U $, so the conservativity condition
  \eqref{eqn:conservativity} is trivial. Hence, there are no
  constraints on
  $ ( u , \sigma , \widehat{ u } , \widehat{ \sigma } ) $ other than
  the local conditions
  \eqref{eqn:localSolverA}--\eqref{eqn:numericalFlux} on each
  triangle, as discussed in \autoref{ex:laplaceCG-H}.

  From the calculation in \eqref{eqn:cg-hBoundaryIntegrals}, we have
  \begin{align*}
    \int _{ e _{1 2} } \mathrm{d} u \wedge \mathrm{d} w + \int _{ e _{ 1 3 } } \mathrm{d} u \wedge \mathrm{d} w
    &= \frac{ \sqrt{ 3 } }{ 6 } ( \mathrm{d} u _3 \wedge \mathrm{d} u _1 - \mathrm{d} u _1 \wedge \mathrm{d} u _2  ) \\
    &= \frac{ \sqrt{ 3 } }{ 6 } ( \mathrm{d} u _2 + \mathrm{d} u _3 ) \wedge \mathrm{d} u _1  ,
  \end{align*}
  and similarly,
  \begin{equation*}
    \int _{ e _{3 4} } \mathrm{d} u \wedge \mathrm{d} w + \int _{ e _{ 2 4 } } \mathrm{d} u \wedge \mathrm{d} w = \frac{ \sqrt{ 3 } }{ 6 } ( \mathrm{d} u _2 + \mathrm{d} u _3 ) \wedge \mathrm{d} u _4 ,
  \end{equation*}
  so adding these gives
  \begin{equation}
    \label{eqn:cghBoundaryIntegral}
    \int _{ \partial U } \mathrm{d} u \wedge \mathrm{d} w = \frac{ \sqrt{ 3 } }{ 6 } ( \mathrm{d} u _2 + \mathrm{d} u _3 ) \wedge ( \mathrm{d} u _1 + \mathrm{d} u _4 ) \neq 0 .
  \end{equation} 
  We conclude that this is nonzero since $ u _1 , \ldots , u _4 $ are
  independent degrees of freedom that may be varied independently. In
  particular, if we take variations
  $ ( v , \operatorname{grad} v, v, \widehat{ \tau } ) $ and
  $ ( v ^\prime , \operatorname{grad} v ^\prime , v ^\prime ,
  \widehat{ \tau } ^\prime ) $ with
  \begin{equation*}
    v _2 = v ^\prime _1 = 1 , \qquad v _1 = v _3 = v _4 = v _2 ^\prime = v _3 ^\prime = v _4 ^\prime = 0 , 
  \end{equation*}
  we have
  \begin{align*}
    \int _{ \partial U } ( v \widehat{ \tau } ^\prime - v ^\prime \widehat{ \tau } ) \cdot \mathbf{n}
    = \frac{ \sqrt{ 3 }}{ 6 } \bigl[ ( v _2 + v _3 )( v _1 ^\prime + v _4 ^\prime ) - ( v _2 ^\prime + v _3 ^\prime ) ( v _1 + v _4 ) \bigr]
    = \frac{ \sqrt{ 3 }}{ 6 } \neq 0 .
  \end{align*} 
  Hence, the CG-H method is not strongly multisymplectic.
\end{proof}

\begin{remark}
  The failure of strong multisymplecticity, in this example, may also
  be seen via the failure of the bilinear form
  $ ( v, v ^\prime ) \mapsto \int _{ \partial U } v \widehat{ \tau }
  ^\prime \cdot \mathbf{n} $ to be symmetric. Indeed, if we write
  $ \mathbf{v} = ( v _1, v _2 , v _3 , v _4 ) $ and
  $ \mathbf{v} ^\prime = ( v _1 ^\prime, v _2 ^\prime , v _3 ^\prime ,
  v _4 ^\prime ) $, then we may represent this as the quadratic form,
  \begin{equation*}
    \int _{ \partial U } v \widehat{ \tau } ^\prime \cdot \mathbf{n} = \frac{ \sqrt{ 3 } }{ 6 } \mathbf{v} ^T \begin{bmatrix}
      2 & -1 & -1 & 0 \\
      0 & 2 & -2 & 0 \\
      0 & -2 & 2 & 0 \\
      0 & -1 & -1 & 2 
    \end{bmatrix}
    \mathbf{v} ^\prime , 
  \end{equation*}
  which is immediately seen not to be symmetric. Neglecting the scalar
  factor of $ \sqrt{ 3 } / 6 $, the antisymmetrization of this matrix
  is
  \begin{align*}
    \begin{bmatrix}
      2 & -1 & -1 & 0 \\
      0 & 2 & -2 & 0 \\
      0 & -2 & 2 & 0 \\
      0 & -1 & -1 & 2
    \end{bmatrix} - \begin{bmatrix}
      2 & -1 & -1 & 0 \\
      0 & 2 & -2 & 0 \\
      0 & -2 & 2 & 0 \\
      0 & -1 & -1 & 2
    \end{bmatrix} ^T &=
                       \begin{bmatrix}
                         0 & - 1 & - 1 & 0 \\
                         1 & 0 & 0 & 1 \\
                         1 & 0 & 0 & 1 \\
                         0 & - 1 & - 1 & 0 \\
                       \end{bmatrix} \\
        &=
          \begin{bmatrix}
            0 \\
            1 \\
            1 \\
            0
          \end{bmatrix} \otimes
    \begin{bmatrix}
      1 \\ 0 \\ 0 \\ 1
    \end{bmatrix}
    -  
    \begin{bmatrix}
      1 \\ 0 \\ 0 \\ 1
    \end{bmatrix}
    \otimes 
    \begin{bmatrix}
      0 \\
      1 \\
      1 \\
      0
    \end{bmatrix}.
  \end{align*}
  This is precisely the matrix corresponding to
  \begin{equation*}
    ( \mathrm{d} u _2 + \mathrm{d} u _3 ) \otimes ( \mathrm{d} u _1 + \mathrm{d} u _4 ) - ( \mathrm{d} u _1 + \mathrm{d} u _4 ) \otimes ( \mathrm{d} u _2 + \mathrm{d} u _3 ) = ( \mathrm{d} u _2 + \mathrm{d} u _3 ) \wedge ( \mathrm{d} u _1 + \mathrm{d} u _4 ) ,
  \end{equation*}
  which agrees with \eqref{eqn:cghBoundaryIntegral}.
\end{remark}

\begin{remark}
  \citet{CoGoLa2009} observe that CG-H can be seen as a limiting case
  of the LDG-H method \eqref{eqn:ldg3} where the penalty
  $ \lambda \equiv + \infty $. They also consider more general LDG-H
  methods where $ \lambda $ is infinite on some elements and finite on
  others \citep[Section 3.4]{CoGoLa2009}. However, allowing the
  penalty to be infinite imposes continuity conditions on
  $ \widehat{ V } $, meaning that conservativity (and thus
  multisymplecticity) is only weak rather than strong.
\end{remark}

\subsection{The NC-H method}

The hybridized nonconforming (NC-H) method uses the local function
spaces
\begin{equation*}
  V (K) = \bigl[ \mathcal{P} _r (K) \bigr] ^n , \qquad \Sigma (K) = \bigl[ \mathcal{P} _{ r -1 } (K) \bigr] ^{ m n } ,
\end{equation*}
just as in the CG-H method. The trace spaces are
\begin{equation*}
  \widehat{ V } = \bigl[ \mathcal{P} _{r-1} ( \mathcal{E} _h ) \bigr] ^n , \qquad \widehat{ \Sigma } ( \partial K ) = \bigl\{ \widehat{ w }  \mathbf{n} : \widehat{ w } \rvert _e  \in [ \mathcal{P} _{ r -1 } (e) ] ^n,\ \forall e \in \partial K  \bigr\} ,
\end{equation*}
and the local flux functions are
\begin{equation*}
  \Phi _K ( u, \sigma, \widehat{ u } , \widehat{ \sigma } ) = ( \widehat{ u } - u ) \mathbf{n} .
\end{equation*}
Unlike CG-H,
$ ( \widehat{ u } - u ) \mathbf{n} \rvert _{ \partial K } $ is
generally not in $ \widehat{ \Sigma } ( \partial K ) $, so we cannot
conclude that $ \widehat{ u } $ equals $u$ on $ \partial K $, except
in the weak sense of \eqref{eqn:numericalFlux}.

\begin{theorem}
  The NC-H method is strongly multisymplectic.
\end{theorem}

\begin{proof}
  The flux $ \Phi _K $ is that considered in \autoref{thm:ncFlux}, so
  it suffices to show that the hypothesis of that theorem holds, i.e.:
  for all $ \tau \in \Sigma (K) $, there exists
  $ \widehat{ \tau } \in \widehat{ \Sigma } ( \partial K ) $ such that
  $ \widehat{ \tau } _i ^\mu \mathbf{n} _\mu = \tau _i ^\mu \mathbf{n}
  _\mu \rvert _{ \partial K } $ for $ i = 1, \ldots, n $. Given
  $ \tau \in \Sigma (K) $, this condition is satisfied by
  $ \widehat{ \tau } _i ^\mu = \tau _i ^\nu \mathbf{n} _\nu \mathbf{n}
  ^\mu \rvert _{ \partial K } $, i.e., the projection of
  $ \tau \rvert _{ \partial K } $ onto the unit normal, so the NC-H
  method is multisymplectic.

  Next, if $ e = \partial K ^+ \cap \partial K ^- $ is an interior
  facet, then the definition of
  $ \widehat{ \Sigma } ( \partial K ^\pm ) $ implies that
  $ \widehat{ \sigma } \rvert _{ e ^\pm } \in \bigl[ \mathcal{P}
  _{r-1} ( e )\bigr] ^n $, so
  $ \llbracket \widehat{ \sigma } \rrbracket \rvert _e \in \bigl[
  \mathcal{P} _{ r -1 } (e) \bigr] ^n $. Since this holds for all
  $ e \in \mathcal{E} _h ^\circ $, the extension by zero of
  $ \llbracket \widehat{ \sigma } \rrbracket $ to $ \mathcal{E} _h $
  is in $ \widehat{ V } _0 $. Therefore, the NC-H method is strongly
  conservative, so \autoref{thm:strongMSCL} implies that it is strongly
  multisymplectic.
\end{proof}

\subsection{The IP-H methods}
We finally consider the special case of the hybridized interior
penalty method (IP-H), which---unlike the methods considered
above---is somewhat idiosyncratic to semilinear elliptic
systems. Consider a system of the form
\begin{equation}
  \label{eqn:semilinearSystem}
  \partial _\mu u ^i = a _{ \mu \nu } ^{ i j } \sigma _j ^\nu , \qquad -\partial _\mu \sigma _i ^\mu = \frac{ \partial F }{ \partial u ^i } ,
\end{equation}
which generalizes the scalar ($ n = 1 $) semilinear PDEs discussed in
\autoref{sec:pdes} to $ n \geq 1 $. Here,
$ a = a ^{ \mu \nu } _{ i j } (x) $ is a symmetric, positive-definite
$ m n \times m n $ matrix with inverse
$ a ^{ i j } _{ \mu \nu } (x) \coloneqq \bigl( a _{ i j } ^{ \mu \nu }
(x) \bigr) ^{-1} $. These are the de~Donder--Weyl equations for the
Hamiltonian
\begin{equation*}
  H ( x , u , \sigma ) = \frac{1}{2} a _{ \mu \nu } ^{ i j } (x) \sigma _i ^\mu \sigma _j ^\nu + F ( x, u ) ,
\end{equation*}
so in particular \eqref{eqn:semilinearSystem} is a canonical
multisymplectic system of PDEs.

For such a system, the IP-H method uses the local function spaces
\begin{subequations}
  \begin{equation}
    \label{eqn:ip-h}
    V (K) = \bigl[ \mathcal{P} _{ r } (K) \bigr] ^n , \qquad \Sigma (K) = \bigl[ \mathcal{P} _r (K) \bigr] ^{m n} .
  \end{equation}
  We also consider the ``IP-H-like'' method, suggested by
  \citet[p.~1351]{CoGoLa2009}, which uses the function spaces
  \begin{equation}
    \label{eqn:ip-h-like}
    V (K) = \bigl[ \mathcal{P} _{ r } (K) \bigr] ^n , \qquad \Sigma (K) = \bigl[ \mathcal{P} _{r-1} (K) \bigr] ^{m n} .
  \end{equation} 
\end{subequations}
For both methods, the trace spaces are taken to be
\begin{equation*}
  \widehat{ V } = \bigl[ \mathcal{P} _r ( \mathcal{E} _h ) \bigr] ^n , \qquad \widehat{ \Sigma } ( \partial K ) = \bigl[ L ^2 ( \partial K ) \bigr] ^{ m n } .
\end{equation*}
Based on the observation that the classical solution satisfies
$ \sigma _i ^\mu = ( a \operatorname{grad} u ) _i ^\mu = a _{ i j } ^{
  \mu \nu } \partial _\nu u ^j $, these methods take the local flux
functions to be
\begin{equation*}
  \Phi _K ( u , \sigma , \widehat{ u } , \widehat{ \sigma } ) = ( \widehat{ \sigma } - a \operatorname{grad} u  ) - \lambda ( \widehat{ u } - u ) \mathbf{n} .
\end{equation*}
Thus, the IP-H and ``IP-H-like'' methods are essentially the LDG-H
methods \eqref{eqn:ldg2} and \eqref{eqn:ldg3}, respectively, except
with $\sigma$ replaced by $ a \operatorname{grad} u $ in the local
flux functions. As with LDG-H methods, the penalty function
$ \lambda \rvert _{ \partial K } $ is taken to be piecewise constant
on $ \partial K $ for each $ K \in \mathcal{T} _h $, and on internal
facets $ e = \partial K ^+ \cap \partial K ^- $, the constants
$ \lambda \rvert _{ e ^\pm } $ need not be equal to one another.  The
flux condition \eqref{eqn:numericalFlux} states that
\begin{equation*}
  \widehat{ \sigma } _i ^\mu \rvert _{ \partial K } = \bigl[ a _{ i j } ^{ \mu \nu } \partial _\nu u ^j + \lambda \delta _{ i j } ( \widehat{ u } ^j - u ^j ) \mathbf{n} ^\mu \bigr] \rvert _{ \partial K } ,
\end{equation*}
so we may eliminate \eqref{eqn:numericalFlux} and substitute the
expression on the right-hand side wherever $ \widehat{ \sigma } $
appears in the remaining equations. We also assume, following
\citet{CoGoLa2009}, that $a$ is constant on each
$ K \in \mathcal{T} _h $.

\begin{theorem}
  \label{thm:ip}
  For the semilinear system \eqref{eqn:semilinearSystem} where $a$ is
  constant on each $ K \in \mathcal{T} _h $, the IP-H and
  ``IP-H-like'' methods are strongly multisymplectic.
\end{theorem}

\begin{proof}
  To prove that the methods are multisymplectic, \autoref{lem:jump}
  states that it suffices to show that \eqref{eqn:msclJump} is
  satisfied on each $ K \in \mathcal{T} _h $. Observe that
  \begin{equation*}
    ( \widehat{ \sigma } _i ^\mu - \sigma _i ^\mu) \rvert _{ \partial K }  = \bigl[ ( a _{ i j } ^{ \mu \nu } \partial _\nu u ^j - \sigma _i ^\mu  ) + \lambda \delta _{ i j } ( \widehat{ u } ^j - u ^j ) \mathbf{n} ^\mu \bigr] \rvert _{ \partial K } .
  \end{equation*}
  We have previously seen that
  \begin{equation*}
    \int _{ \partial K } \lambda \delta _{ i j } \bigl[ \mathrm{d} ( \widehat{ u } ^i - u ^i )
    \wedge \mathrm{d} ( \widehat{ u } ^j - u ^j  )
    \bigr] \mathbf{n} ^\mu  \,\mathrm{d}^{m-1} x_\mu = 0 ,
  \end{equation*}
  by the symmetry of $ \delta $ and the antisymmetry of $ \wedge
  $. Therefore, it remains to show that the terms involving
  $ a _{ i j } ^{ \mu \nu } \partial _\nu u ^j - \sigma _i ^\mu $
  vanish as well.  Integrating \eqref{eqn:localSolverA} by parts gives
  \begin{align*}
    \int _{ \partial K } ( \widehat{ u } ^i - u ^i ) \tau _i ^\mu \,\mathrm{d}^{m-1} x_\mu
    &= \int _K ( \phi ^i _\mu  - \partial _\mu u ^i ) \tau _i ^\mu \,\mathrm{d}^m x \\
    &= \int _K ( a ^{ i j } _{ \mu \nu } \sigma _j ^\nu   - \partial _\mu u ^i ) \tau _i ^\mu \,\mathrm{d}^m x ,
  \end{align*}
  and since this holds for all $ \tau \in \Sigma (K) $, we can write
  \begin{equation*}
    \int _{ \partial K } \bigl[ ( \widehat{ u } ^i - u
    ^i ) \,\mathrm{d} \sigma _i ^\mu \bigr] \,\mathrm{d}^{m-1} x_\mu = \int _K \bigl[ ( a ^{ i j } _{ \mu \nu } \sigma _j ^\nu - \partial _\mu u ^i ) \,\mathrm{d}\sigma _i ^\mu \bigr] \,\mathrm{d}^m x .
  \end{equation*}
  Taking the exterior derivative of both sides yields
  \begin{align*}
    \int _{ \partial K } \bigl[ \mathrm{d} ( \widehat{ u } ^i - u
    ^i ) \wedge \mathrm{d} \sigma _i ^\mu \bigr] \,\mathrm{d}^{m-1} x_\mu
    &= \int _K \bigl[ a ^{ i j } _{ \mu \nu } \,\mathrm{d}\sigma _j ^\nu \wedge \mathrm{d} \sigma _i ^\mu - \mathrm{d} ( \partial _\mu u ^i ) \wedge \mathrm{d} \sigma _i ^\mu \bigr] \,\mathrm{d}^m x \\
    &= \int _K \bigl[ -  \mathrm{d} (  \partial _\mu u ^i ) \wedge \mathrm{d} \sigma _i ^\mu \bigr] \,\mathrm{d}^m x,
  \end{align*}
  where
  $ a ^{ i j } _{ \mu \nu } \,\mathrm{d}\sigma _j ^\nu \wedge
  \mathrm{d} \sigma _i ^\mu = 0 $ by the symmetry of $a$ and the
  antisymmetry of $ \wedge $.  On the other hand, taking
  $ \tau = a \operatorname{grad} v $ for $ v \in V (K) $, we have
  \begin{align*}
    \int _{ \partial K } ( \widehat{ u } ^i - u ^i ) \tau _i ^\mu \,\mathrm{d}^{m-1} x_\mu
    &= \int _K ( a ^{ i j } _{ \mu \nu } \sigma _j ^\nu - \partial _\mu u ^i ) a _{ i k } ^{ \mu \lambda } \partial _\lambda v ^k \,\mathrm{d}^m x \\
    &= \int _K ( \sigma _i ^\mu - a ^{ \mu \nu } _{ i j } \partial _\nu u ^j ) \partial _\mu v ^i \,\mathrm{d}^m x ,
  \end{align*}
  and since this holds for all $ v \in V (K) $, we can write
  \begin{equation*}
    \int _{ \partial K } \bigl[ ( \widehat{ u } ^i - u ^i ) \,\mathrm{d} ( a _{ i j } ^{ \mu \nu } \partial _\nu  u ^j ) \bigr] \,\mathrm{d}^{m-1} x_\mu = \int _K \bigl[ ( \sigma _i ^\mu - a _{ i j } ^{ \mu \nu } \partial _\nu u ^j ) \,\mathrm{d}  ( \partial _\mu u ^i ) \bigr] \,\mathrm{d}^m x .
  \end{equation*} 
  Note that we need $a$ to be constant on $K$ in order for
  $ \tau = a \operatorname{grad} v \in \bigl[ \mathcal{P} _{ r -1 }
  (K) \bigr] ^{ m n } \subset \Sigma (K) $ to be an admissible test
  function for both methods \eqref{eqn:ip-h}--\eqref{eqn:ip-h-like}.
  Taking the exterior derivative of both sides yields
  \begin{multline*}
    \int _{ \partial K } \bigl[ \mathrm{d} ( \widehat{ u } ^i - u ^i ) \wedge \mathrm{d} ( a _{ i j } ^{ \mu \nu } \partial _\nu u ^j) \bigr] \,\mathrm{d}^{m-1} x_\mu\\
    \begin{aligned}
      &= \int _K \bigl[ \mathrm{d} \sigma _i ^\mu \wedge \mathrm{d} ( \partial _\mu u ^i ) - a _{ i j } ^{ \mu \nu } \mathrm{d} ( \partial _\nu u ^j ) \wedge \mathrm{d} (\partial _\mu u ^i ) \bigr] \,\mathrm{d}^m x \\
      &= \int _K \bigl[ \mathrm{d} \sigma _i ^\mu \wedge \mathrm{d} ( \partial _\mu  u ^i ) \bigr] \,\mathrm{d}^m x \\
      &= \int _{ \partial K } \bigl[ \mathrm{d} ( \widehat{ u } ^i - u
      ^i ) \wedge \mathrm{d} \sigma _i ^\mu \bigr] \,\mathrm{d}^{m-1}
      x_\mu,
  \end{aligned}
\end{multline*}
where again we have used the symmetry of $a$ and the antisymmetry of
$ \wedge $. Rearranging this last equality yields
\begin{equation*}
  \int _{ \partial K } \bigl[ \mathrm{d} ( \widehat{ u } ^i - u
  ^i ) \wedge \mathrm{d} ( a _{ i j } ^{ \mu \nu
  } \partial _\nu u ^j - \sigma _i ^\mu ) \bigr] \,\mathrm{d}^{m-1} x_\mu = 0 ,
\end{equation*}
which is precisely \eqref{eqn:msclJump}. Hence, the IP-H method is
multisymplectic.

To show strong multisymplecticity, let
$ e = \partial K ^+ \cap \partial K ^- $ be an internal facet. Then
\begin{equation*}
  \widehat{ \sigma } \rvert _{ e^\pm } = \bigl[ a \operatorname{grad} u + \lambda ( \widehat{ u } - u ) \mathbf{n} \bigr] \rvert _{ e ^\pm } \in \bigl[ \mathcal{P} _r (e) \bigr] ^{ m n } ,
\end{equation*}
so
$ \llbracket \widehat{ \sigma } \rrbracket \rvert _e \in \bigl[
\mathcal{P} _r (e) \bigr] ^n $. (Here, we use that
$ a \rvert _{ e ^\pm } $ and $ \lambda \rvert _{ e ^\pm } $ are
constant.) Hence, $ \llbracket \widehat{ \sigma } \rrbracket $ may be
extended by zero to an element of $ \widehat{ V } _0 $, so both
methods are strongly conservative and thus, by
\autoref{thm:strongMSCL}, strongly multisymplectic.
\end{proof}

\section{Conclusion}

We have generalized the flux formulation and HDG framework of
\citet{CoGoLa2009} to the much larger family of canonical systems of
PDEs \eqref{eqn:pdeSystem}, which includes nonlinear systems. Within
this framework, we have established the multisymplecticity of several
HDG methods, when such methods are applied to canonical Hamiltonian
systems \eqref{eqn:ddw}. These methods include ``hybridized'' versions
of several widely-used classes of finite element methods, suggesting
that---when multisymplectic structure preservation is desired---these
general purpose methods may be used instead of the specialized methods
constructed in previous work, which are often limited in their order
of accuracy and/or use on unstructured meshes.

It is perhaps not surprising that so many finite element methods are
multisymplectic. Indeed, like the Ritz--Galerkin method, the
multisymplectic conservation law is intimately related to variational
principles
(cf.~\citet{LaSnTu1975,KiTu1979,MaPaSh1998,GoIsMaMo2004,GoIsMa2004,VaLiLe2013}). However,
any construction of multisymplectic finite element methods must
address two difficulties: first, that the multisymplectic conservation
law holds for smooth solutions, while finite element spaces are
nonsmooth; and second, that finite element basis functions can have
support on several elements, posing an obstacle to expressing a
localized, per-element conservation law. The flux formulation of
\autoref{sec:flux} provides a solution to both of these
difficulties. Hybridization (i.e., introducing separate spaces of
boundary traces and fluxes) provides a way to express the
multisymplectic conservation law on individual elements, while the
exact flux formulation of \autoref{sec:exactSolutions} provides a
bridge between the smooth and non-smooth cases (particularly
\autoref{thm:exactSolution} and \autoref{cor:exactMSCL}).

There are two final points we wish to emphasize about these results,
by comparison to the previous work discussed in \autoref{sec:intro}.
\begin{enumerate}
\item Multisymplectic HDG methods may be applied to systems of the
  form \eqref{eqn:pdeSystem}, whether or not the user is aware of any
  canonical multisymplectic/Hamiltonian structure. Yet, if such a
  structure is present, it will automatically be preserved. This is
  analogous to the case for certain classes of symplectic integrators
  for ODEs (e.g., symplectic partitioned Runge--Kutta methods), but
  contrasts with previously-studied multisymplectic methods for PDEs
  on unstructured meshes (e.g., Lagrangian variational methods).

\item Many of the previously-constructed multisymplectic methods
  merely satisfy a ``discrete version'' of the multisymplectic
  conservation law, e.g., a finite-difference version of
  \eqref{eqn:msclDifferential} on a lattice. This is true even for the
  previous work on multisymplectic finite element methods
  \citep{GuJiLiWu2001,ZhBaLiWu2003,Chen2008}, which are only
  ``multisymplectic'' in a finite-difference sense on the lattice of
  degrees of freedom. Moreover, this discrete multisymplectic
  conservation law may differ from method to method, depending on how
  the divergence operator is discretized.

  By contrast, the multisymplectic conservation law
  \eqref{eqn:msclHDG} satisfied by these HDG methods is \emph{exactly}
  the integral multisymplectic conservation law
  \eqref{eqn:msclIntegral}, restricted to the numerical traces and
  fluxes on $ \partial K $ for each element $ K \in \mathcal{T} _h $,
  while \eqref{eqn:strongMSCL} extends this to arbitrary unions of
  elements.  Furthermore, this multisymplectic conservation law does
  not differ from method to method: it has precisely the same form for
  weak solutions, in the exact flux formulation of
  \autoref{sec:exactSolutions}, as it does for each of the HDG methods
  of \autoref{sec:hdg}.
\end{enumerate}

One direction for future work is the application of multisymplectic
HDG methods to Hamiltonian time-evolution PDEs. There are two natural
ways in which such work might proceed. First, a multisymplectic HDG
method might be used to semidiscretize the PDE in space, resulting in
a finite-dimensional system of ODEs.  \citet{SaCiNgPeCo2017} have
recently shown that, when an {LDG\nobreakdash-H} method is used to
semidiscretize the acoustic wave equation, the resulting system of
ODEs is Hamiltonian, and one may then apply a symplectic integrator in
time. Second, one might consider the application of spacetime HDG
methods, as in \citet{RhCo2012,RhCo2013,GrMo2014}. In this approach,
one would simultaneously discretize space and time by applying the
flux formulation of \autoref{sec:flux} to the case where $U$ is a
spacetime domain and $ \mathcal{T} _h $ a decomposition into spacetime
elements $ K \in \mathcal{T} _h $. The results of \autoref{sec:flux}
are formulated in sufficient generality to include such methods, so
one might apply them to investigate the multisymplecticity of specific
spacetime HDG methods.

\subsection*{Acknowledgments}

This research was supported in part by the Marsden Fund of the Royal
Society of New Zealand and by the Simons Foundation (award \#279968 to
Ari Stern). We also wish to thank the anonymous referees for their
helpful comments and suggestions.

\end{document}